\documentclass[11pt,reqno]{amsart}

\usepackage[utf8x]{inputenc}
\usepackage[english]{babel}

\usepackage{graphicx}
\usepackage{mathtools}

\usepackage{amssymb,enumerate,amsmath}
\usepackage[bookmarks,colorlinks,citecolor=magenta,pdftex]{hyperref}
\usepackage{tikz}
\usepackage{array}
\usepackage{mathrsfs}
\usepackage{float}
\usepackage{algorithmic}
\usepackage{dsfont}
\usepackage{paralist}
\usepackage[ruled]{algorithm}
\usepackage[a4paper,margin=2.5cm]{geometry}

\newcommand\Defn[1]{\textbf{\color{black}#1}}

\definecolor{light-gray}{gray}{0.8}

\newcommand\R{\mathbb{R}}

\newcommand\rank{\mathrm{rank}}

\newcommand\GGen{\mathsf{Gen}}
\newcommand\Gen{\rank_\GGen}
\newcommand\VGen{\mathcal{V}^\GGen}

\newcommand\LLev{\mathsf{Lev}}
\newcommand\Lev{\rank_\LLev}
\newcommand\VLevels{\mathcal{V}^\LLev}

\newcommand\TTh{\mathsf{Th}}		 
\newcommand\Th{\rank_\TTh}		 
\newcommand\VTheta{\mathcal{V}^\TTh}

\newcommand\PPsd{\mathsf{Psd}} 
\newcommand\Psd{\rank_\PPsd}
\newcommand\VPsd{\mathcal{V}^\PPsd}
\newcommand\VPmin{\mathcal{V}^\PPsd_\mathsf{min}}

\newcommand\oF{\overline{F}}

\renewcommand\emptyset{\varnothing}

\newcommand\Hrk{\rank_{\sqrt{\ }}}

\newcommand\0{\mathbf{0}}
\newcommand\1{\mathbf{1}}

\newcommand\Char[1]{\1_{#1}}

\newcommand\GSeriesPar{\mathcal{G}_{\mathsf{SP}}}

\newcommand\Matroids{\mathcal{M}}
\newcommand\MLevels{\Matroids^\LLev}
\newcommand\MTheta{\Matroids^\TTh}
\newcommand\MGen{\Matroids^\GGen}

\newcommand\BB{\mathcal{B}}
\newcommand\MM{M}
\newcommand\NN{N}
\newcommand\SeriesConn{\mathcal{S}}

\newcommand\x{\mathbf{x}}

\newcommand\GLevels{\mathcal{G}^\mathsf{Lev}}
\newcommand\GTheta{\mathcal{G}^\TTh}

\DeclareMathOperator{\MG}{\MM(G)}
\DeclareMathOperator{\rk}{\mathbf{rk}}%
\DeclareMathOperator{\conv}{\mathrm{conv}}
\DeclareMathOperator{\cone}{cone}

\newcommand\cF{\overline{F}}
\newcommand\Fme{\hat{F}}

\numberwithin{equation}{section}
\newtheorem{thm}{Theorem}[section]
\newtheorem{cor}[thm]{Corollary}
\newtheorem{prop}[thm]{Proposition}
\newtheorem{lemma}[thm]{Lemma}

\theoremstyle{definition}
\newtheorem{defi}[thm]{Definition}

\newtheorem{ex}{Example}

\title{Theta rank, levelness, and matroid minors}

\author{Francesco Grande}
\address{Fachbereich Mathematik und Informatik, %
Freie Universit\"at Berlin, Berlin, %
Germany}
\email{fgrande@math.fu-berlin.de}

\author{Raman Sanyal}
\address{Institut f\"ur Mathematik, Goethe-Universit\"at Frankfurt, Germany}
\email{sanyal@math.uni-frankfurt.de}


\date{\today}
\thanks{F.~Grande was supported by DFG within the research training group
``Methods for Discrete Structures'' (GRK1408).
R.~Sanyal was supported by the European Research Council under the
European Union's Seventh Framework Programme (FP7/2007-2013) / ERC grant
agreement n$^\mathrm{o}$ 247029 and the DFG
Collaborative Research Center SFB/TR 109 ``Discretization
in Geometry and Dynamics''.}

\begin{document}

\maketitle

\begin{abstract}
    The Theta rank of a finite point configuration $V$ is the maximal degree
    necessary for a sum-of-squares representation of a non-negative affine
    function on $V$. This is an important invariant for polynomial
    optimization that is in general hard to determine. We study the Theta rank
    and levelness, a related discrete-geometric invariant, for matroid base
    configurations. It is shown that the class of matroids with bounded Theta
    rank or levelness is closed under taking minors. This allows for a
    characterization of matroids with bounded Theta rank or levelness in terms
    of forbidden minors.  We give the complete (finite) list of excluded
    minors for Theta-$1$ matroids which generalizes the well-known
    series-parallel graphs.  Moreover, the class of Theta-$1$ matroids can be
    characterized in terms of the degree of generation of the vanishing ideal
    and in terms of the psd rank for the associated matroid base polytope.  We
    further give a finite list of excluded minors for $k$-level graphs and
    matroids and we investigate the graphs of Theta rank~$2$.
\end{abstract}

\section{Introduction}
\label{sec:intro}

Let $V$ be a configuration of finitely many points in $\R^n$. An affine
function $\ell(\x) = \delta - \langle c, \x \rangle$ that only takes
non-negative values on 
$V$ is called $\mathbf{k}$\Defn{-sos} with respect to $V$ if there exist
polynomials ${h_1,\ldots,h_s\in \mathbb{R}[x_1,\ldots , x_n]}$ such that $\deg
h_i\leq k$ and
\begin{equation}\label{eqn:k-sos}
    \ell(v) \ = \ 
    h_1^2(v) + 
    h_2^2(v) + 
    \cdots + 
    h_s^2(v) 
\end{equation}
for all $v \in V$. The \Defn{Theta rank} $\Th(V)$ of $V$ is the smallest $k
\ge 0$ such that every non-negative affine function on $V$ is $k$-sos. The
Theta rank was introduced in~\cite{GPT} as a measure for the `complexity' of
linear optimization over $V$ using tools from polynomial optimization. If $V$
is given as the solutions to a system of polynomial equations, then the size
of a semidefinite program for the (exact) optimization of a linear function
over $V$ is of order  $O(n^{\Th(V)})$.  For many practical applications, for
example in combinatorial optimization, an algebraic description of $V$ is
readily available and the semidefinite programming approach is the method of
choice. Clearly, situations with high Theta rank render the approach
impractical. We are interested in
\[
    \VTheta_{k} \ := \ \{ V \text{ point configuration\,} : \Th(V) \le k \}.
\]
As $V$ is finite and $\ell(\x)$ non-negative on $V$, we may interpolate
$\sqrt{\ell(\x)}$ over $V$ by a single polynomial which shows that $\Th(V) \le
|V|-1$; cf.~\cite[Rem.~4.3]{GPT}. This, however, is a rather crude estimate
as the $0/1$-cube $V = \{0,1\}^n$ has Theta rank $1$.  

Let $\ell(\x)$ be a non-negative affine function on $V$. The subconfiguration
$V^\prime = \{ v \in V : \ell(v)=0\}$ is called a \Defn{face} of $V$ with supporting
hyperplane $H = \{ \x \in \R^n :\ell(\x) = 0\}$.  If $V^\prime \neq V$  is
inclusion-maximal, then $V^\prime$ is called a \Defn{facet} and $H$ (and
equivalently $\ell(\x)$) \Defn{facet-defining}. If $V$ is a full-dimensional
point configuration then $H$ and $\ell(\x)$, up to positive scaling, are unique.
It follows from basic convexity that $\Th(V)$ is the smallest $k$ such that
all facet-defining affine functions $\ell(\x)$ are $k$-sos. 
A point configuration $V$ is \Defn{$k$-level} if for every facet-defining
hyperplane $H$ there are $k$ parallel hyperplanes $H = H_1,H_2,\dots,H_k$ with 
\[
    V \ \subseteq \ H_1 \cup H_2 \cup \cdots \cup H_k.
\]
Equivalently, $V$ is $k$-level if every facet-defining affine function
$\ell(\x)$ takes at most $k$ distinct values on $V$. 
The
\Defn{levelness} $\Lev(V)$ of $V$ is the smallest $k$ such that $V$ is
$k$-level. Using polynomial interpolation as before, it is easy to see that
$\Th(V) \le \Lev(V) - 1$; see~\cite[Rem.~4.3]{GPT}.  Hence, the class
$\VLevels_k$ of all $k$-level point configurations is a subclass of
$\VTheta_{k-1}$. A main result of~\cite{GPT} is the following characterization
of $\VTheta_1$.
\begin{thm}[{\cite[Thm.~4.2]{GPT}}]\label{thm:GPT}
    Let $V$ be a finite point configuration. Then $V$ has Theta rank $1$ if
    and only if $V$ is $2$-level.
\end{thm}

For $k \ge 2$, it can be shown that $\VLevels_k\subsetneq \VTheta_{k-1}$.  The
(convex) polytopes $P = \conv(V)$ for $2$-level point configurations are very
interesting. They arise in the study of extremal centrally-symmetric
polytopes~\cite{SWZ} as well as in statistics under the name of
\emph{compressed} polytopes~\cite{Sullivant}. Every $2$-level polytope is
affinely isomorphic to a $0/1$-polytope which gives them a combinatorial
character. Nevertheless we lack a genuine understanding of this class of
polytopes.

In this paper, we study the subclasses $\MTheta_k \subset \VTheta_k$ of point
configurations coming from the bases of \emph{matroids}. We recall the notion
of matroids and the associated geometric objects in
Section~\ref{sec:matroids_pts}. In particular, we show that the classes
$\MTheta_k$ are closed under taking \emph{minors}. This, in principle, allows
for a characterization of $\MTheta_k$ in the form of forbidden sub-structures.
In Section~\ref{sec:2level}, we focus on the class $\MTheta_1$ of matroids of
Theta rank $1$ or, equivalently, $2$-level matroids. Our first main result is
the following.

\begin{thm}\label{thm:main}
    Let $\MM = (E,\BB)$ be a matroid and $V_\MM \subset \R^E$ the
    corresponding base configuration. The following are equivalent:
    \begin{enumerate}[\rm (i)]
        \item $V_\MM$ has Theta rank $1$ or, equivalently, is $2$-level;
        \item $\MM$ has no minor isomorphic to $\MM(K_4)$, $\mathcal{W}^3$,
            $Q_6$, or $P_6$;
        \item $\MM$ can be constructed from uniform matroids by taking direct
            sums or $2$-sums;
        \item The vanishing ideal $I(V_\MM)$ is generated in degrees $\le 2$;
        \item The base polytope $P_\MM$ has minimal psd rank.
    \end{enumerate}
\end{thm}

Part (ii) yields a complete and, in particular, finite list of excluded minors
whereas (iii) gives a synthetic description of this class of matroids.  The
parts (iv) and (v) are proven in Section~\ref{sect:gen_psd}. The former states
that $2$-level matroids are precisely those matroids $\MM$ for which the base
configuration $V_\MM$ is cut out by quadrics (Theorem~\ref{thm:main2}). This
contrasts the situation for general point configurations
(Example~\ref{ex:gen2}). The psd rank of a polytope $P$ is the smallest `size'
of a spectrahedron that linearly projects to $P$. The psd rank was studied
in~\cite{GPT2,GRT} and it was shown that the psd rank $\Psd(P)$ is at least
$\dim P + 1$. Part (v) shows that the $2$-level matroids are exactly those
matroids for which the psd rank of the base polytope $P_\MM = \conv(V_\MM)$ is
minimal. Again, this is in strong contrast to the psd rank of general
polytopes.

In Section~\ref{sec:higher} we give a complete list of excluded minors for
$k$-level graphs (Theorem~\ref{thm:excludminorsgraphs}). The classes of
$3$-level and $4$-level graphs appear in works of
Halin~(see~\cite[Ch.~6]{Diestel}) and Oxley~\cite{Oxley5wheel}. In particular,
the wheel with $5$ spokes $W_5$ is shown to have Theta rank $3$. Combined with
results of Oxley~\cite{Oxley5wheel}, this yields a finite list of candidates
for a complete characterization of Theta-$2$ graphs. Whereas the list of
forbidden minors for graphs is always finite, this is generally not true for
matroids. In Section~\ref{sec:klevel} we show that $k$-levelness of matroids
is characterized by finitely many excluded minors and we conjecture this to be
true for matroids of Theta rank~$k$.

\textbf{Acknowledgements.} 
We would like to thank Philipp Rostalski and Frank Vallentin for helpful
discussions regarding computations and we thank Bernd Sturmfels for his
interest in the project. We would also like to thank the two referees for
careful reading and many valuable suggestions.

\section{Point configurations and matroids}
\label{sec:matroids_pts}

In this section we study properties of Theta rank and levelness related
to the geometry of the point configuration.  In particular, we investigate 
the behavior of these invariants under taking sub-configurations. We recall
basic notions from matroid theory and associated point configurations and
polytopes. 

\subsection{Theta rank, levelness, and face-hereditary properties}
The definitions of levelness and Theta rank make only reference to the affine
hull of the configuration $V$ and thus neither depend on the embedding nor on
a choice of coordinates. To have it on record we note the following basic
property.

\begin{prop}\label{prop:TH-aff}%
    The levelness and the Theta rank of a point configuration are invariant
    under affine transformations.
\end{prop}

That this does not hold for (admissible) projective transformations is clear
for the levelness and for the Theta rank follows from Theorem~\ref{thm:GPT}.

\begin{prop}\label{prop:TH-product}%
    Let $V_1 \subset \R^{d_1}$ and $V_2 \subset \R^{d_2}$ be point
    configurations. Then the Theta rank satisfies $\Th(V_1 \times V_2) = \max(\Th(V_1),\Th(V_2))$. The
    same is true for $\Lev(V_1 \times V_2)$.
\end{prop}
\begin{proof}
    \newcommand\y{\mathbf{y}}
    Every facet $F \subset V_1 \times V_2$ is of the form $F = F_1 \times V_2$
    for a facet $F_1 \subset V_1$ or $F = V_1 \times F_2$ for a facet $F_2
    \subset V_2$. If $F = F_1 \times V_2$, then there is a facet-defining
    function $\ell(\x,\y) \in \R[\x,\y]$ that does not depend $\y$ and
    $\ell(\x,0)$ is facet-defining for $F_1 \subset V_1$. Thus, any
    representation~\eqref{eqn:k-sos} of $\ell(\x,0)$ on $V_1$ is already a
    representation for $\ell(\x,\y)$ on $V_1 \times V_2$. The argument for a
    facet of the form $F = V_1 \times F_2$ is identical.
\end{proof}

The Theta rank as well as the levelness of a point configuration are not
monotone with respect to taking subconfigurations as can be seen by removing a
single point from $\{0,1\}^d$. However, it turns out that monotonicity holds
for subconfigurations induced by supporting hyperplanes. Let us call a
collection $\mathcal{P}$ of point configurations \Defn{face-hereditary} if it
is closed under taking faces. That is, 
$V \cap H \in \mathcal{P}$ for any $V \in \mathcal{P}$ and supporting
hyperplane $H$ for $V$.

\begin{lemma}\label{lem:face-hed}
    The classes $\VTheta_{ k}$ and $\VLevels_{ k}$ are face-hereditary.
\end{lemma}
\begin{proof}
    Let $V \subset \R^d$ be a full-dimensional point configuration and $H = \{
    p \in \R^d : g(p) = 0\}$ a supporting hyperplane such that the affine hull
    of $V^\prime := V \cap H$ has codimension $1$.  Let $\ell(\x)$ be an
    affine function that defines a facet of $V'$.  Observe that $\ell(\x)$ and
    $\ell_\delta(\x) := \ell(\x)+ \delta g(\x)$ give the same affine function
    on $V^\prime$ for all $\delta$. For
    \[
        \delta \ = \ \max \bigl\{\tfrac{-\ell(v)}{g(v)} : v  \in V \setminus
        V^\prime\bigr\}
    \]
    $\ell_\delta(\x)$ is non-negative on $V$. Hence any
    representation~\eqref{eqn:k-sos} of $\ell_\delta$ over $V$ yields a
    representation for $\ell$ over $V^\prime$. Moreover, the levelness of
    $\ell_\delta(\x)$ gives an upper bound on the levelness of $\ell(\x)$.
\end{proof}

It is interesting to note that these properties are not hereditary with
respect to arbitrary hyperplanes. Indeed, consider the point configuration
\[
V \ = \ (\{0,1\}^n \times \{-1,0,1\} ) \setminus \{\0\}
\]
It can be easily seen that $\Th(V) = \Lev(V) - 1 = 2$. The hyperplane $H = \{
\x : x_{n+1} = 0\}$ is not supporting and $V^\prime = V \cap H =
\{0,1\}^n \setminus \{\0\}$. The affine function $\ell(\x) = x_1 +
\cdots + x_n -1$ is facet-defining for $V^\prime$ with $n$ levels. As for the
Theta rank, any representation~\eqref{eqn:k-sos} yields a polynomial
$f(\x) = \ell(\x) - \sum_i h_i^2(\x)$ of degree $2k$ that vanishes on
$V^\prime$ and ${f(\0)= -1-\sum_i h_i^2(\0) <0}$. For $n > 4$, the following proposition assures
that $\Th(V^\prime) \ge 3$.

\begin{prop}\label{prop:cube_gen}
    Let $V^\prime = \{0,1\}^n \setminus \{\0\}$ and $f(\x)$ a
    polynomial vanishing on $V^\prime$ and $f(\0)\neq 0$. Then $\deg f \ge n$.
\end{prop}
\begin{proof}
   For a monomial $\x^\alpha$, let $\tau = \{ i : \alpha_i > 0 \}$ be its
   support. Over the set of $0/1$-points it follows that $\x^\alpha$ and
   $\x^\tau := \prod_{i \in \tau} x_i$ represent the same function. Hence, we
   can assume that $f$ is of the form $f(\x) \ = \ \sum_{\tau \subseteq [n]}
   c_\tau \x^\tau$ for some $c_\tau \in \R$, $\tau \subseteq [n]$.  Moreover
   $c_{\emptyset}=f(0)\neq 0$ and without loss of generality we can assume
   $c_{\emptyset}=1$. Any point $v \in V^\prime$ is of the form $v =
   \1_\sigma$ for some $\emptyset \neq \sigma \subseteq [n]$ and we calculate
    \[
        0 \ = \ f(v) \ =\sum_{\emptyset \subseteq \tau \subseteq \sigma} c_\tau.
    \]
    It follows that $c_\tau$ satisfies the defining conditions of the
    \emph{M\"obius function} of the Boolean lattice and hence equals
    $c_\tau=(-1)^{|\tau |}$ for all $\tau \subseteq [n]$. In particular
    $c_{[n]} \neq 0$ which finishes the proof.
\end{proof}

\subsection{Matroids and basis configurations}\label{ssec:matroids}

We now introduce the combinatorial point confi\-gurations that are
our main object of study. Matroids and their combinatorial theory are a vast
subject and we refer the reader to the book by Oxley~\cite{Oxley} for further
information.

\begin{defi}\label{dfn:matroid}
    A \Defn{matroid} of rank $k$ is a pair $\MM = (E,\mathcal{B})$
    consisting of a finite ground set $E$  and a collection of bases
    $\emptyset\neq \mathcal{B}\subseteq  \binom{E}{k}$ satisfying the basis
    exchange axiom: for $B_1,B_2 \in \BB$ and $x \in B_1 \setminus B_2$ there
    is
     $y\in B_2\setminus B_1$ such that $(B_1 \setminus x)\cup y \in \BB$.
\end{defi}

A set $I \subseteq E$ is \Defn{independent} if $I \subseteq B$ for some $B \in
\BB$. The \Defn{rank} of $X$, denoted by $\rk_{\MM}(X)$, is the cardinality of
the largest independent subset contained in $X$. We simply write $\rk(X)$ if
$\MM$ is clear from the context. The rank of $\MM$ is $\rk(\MM) := \rk_\MM(E)$.
The \Defn{circuits} of $\MM$
are the inclusion-minimal dependent subsets. An element $e$ is called a
\Defn{loop} if $\{e\}$ is a circuit. We say that $e,f \in E$ are
\Defn{parallel} if $\{e,f\}$ is a circuit.  A \Defn{parallel class} $H
\subseteq E$ is the equivalence class of elements parallel to each other. The
class $H$ is \Defn{non-trivial} if $|H| > 1$.  A matroid is \Defn{simple} if
it does not contain loops or parallel elements.  A \Defn{flat} of a matroid is
a set $F\subseteq E$ such that $\rk(F)<\rk(F\cup e)$ for all  $e\in E\setminus
F$. 

A particular class of matroids that we will consider are the \Defn{graphic
matroids}. To a graph $G = (V,E)$ we associate the matroid $\MM(G) = (E,
\BB)$. The bases are exactly the spanning forests of $G$. The running example
for this section is the following.

\begin{ex}\label{ex:graph}
    Let $G$ be the graph
    \begin{center}
        \includegraphics[scale=.8]{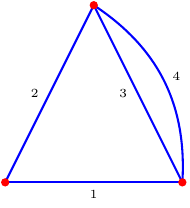}
    \end{center}
    The graphic matroid $\MM = \MM(G)$ has ground
    set $E = \{1,2,3,4\}$, $\rk(\MM)=2$, and bases 
    \[
    \BB(G) \ = \ \{ 12, 13, 14, 23, 24 \}.
    \]
\end{ex}

The \Defn{dual matroid} $\MM^*$ of the matroid $\MM=(E,\BB)$ is the matroid
defined by the pair $(E,\BB^*)$ where $\BB^*=\{E\setminus B \; : \; B\in \BB
\} $. A \Defn{coloop} of $\MM$ is an element which is a loop of $\MM^*$.
Equivalently it is an element which appears in every basis of $\MM$.\\ 
If $e \in E$ is not a coloop, we define the \Defn{deletion} as $\MM\setminus e
:= (E \setminus e, \{ B \in \BB : e \not\in B\})$.  If $e$ is a coloop, then
the bases of $\MM\setminus e$ are $\{ B \setminus e : B \in \BB\}$.  Dually, if
$e \in E$ is not a loop, we define the \Defn{contraction} as $\MM/e := (E
\setminus e, \{ B\setminus e : e \in B \in \BB\})$. These operations can be
extended to subsets $X \subseteq E$ and we write $\MM\setminus X$ and $\MM/ X$,
respectively. We also define the \Defn{restriction} of $\MM$ to a subset $X
\subseteq E$ as $\MM|_X := \MM\setminus(E \backslash X)$. Note that $(\MM \setminus
X)^* = \MM^* / X$. A \Defn{minor} of $\MM$ is a matroid obtained from $\MM$ by
a sequence of deletion and contraction operations. The subclass of graphic
matroids is closed under taking minors but not under taking duals.

To each matroid we associate a point configuration representing the set of
bases. For a fixed ground set $E$ let us write $\Char{X} \in \{0,1\}^E$ for the
characteristic vector of $X \subseteq E$.

\begin{defi}\label{dfn:base_conf}
    Let $\MM = (E,\BB)$ be a matroid. The \Defn{base configuration} of $\MM$ is
    the point configuration
    \[
    V_{\MM} \ := \ \{ \Char{B} : B \in \BB\} \ \subset \ \R^E.
    \]
    The \Defn{base polytope} of $\MM$ is $P_\MM := \conv(V_\MM)$.

\end{defi}

The dual $\MM^*$ is obtained by taking the complements of bases. The
corresponding base configuration is thus 
\begin{equation}\label{eqn:V-dual}
    V_{\MM^*} \ = \ \1 - V_\MM.
\end{equation}
In particular, $V_\MM$ and $V_{\MM^*}$ are related by an affine transformation
and from Proposition~\ref{prop:TH-aff} we deduce the following fact.
\begin{cor}
    For every matroid $\MM$, $V_\MM$ and $V_{\MM^*}$ have the same Theta rank
    and levelness.
\end{cor}

For a point configuration $V \subset \R^m$, we define $\dim V$ to be the
dimension of the affine hull of $V$.  Then $V_\MM$ is not a full-dimensional
point configuration in $\R^E$.  Indeed, $V_\MM$ is contained in the hyperplane
$\sum_{e \in E} x_e = \rk(E)$.  In order to determine the dimension of $V_\MM$
we need to consider the relations among elements of $E$: $e_1,e_2\in E$ are
related if there exists a circuit of $\MM$ containing both. This is an
equivalence relation and the equivalence classes are called the
\Defn{connected components} of $\MM$. Let us write $c(\MM)$ for the number of
connected components. The matroid $\MM$ is \Defn{connected} if $c(\MM)=1$. 
The following basic result due to Tutte will be indispensible throughout this
paper.

\begin{lemma}[{\cite[Thm.~4.3.1]{Oxley}}]\label{lem:tutte}
    Let $\MM = (E,\mathcal{B})$ be a connected matroid and $e \in E$. Then
    $\MM \setminus e$ or $\MM / e$ is connected.
\end{lemma}

Let $\MM_1$ and $\MM_2$ be matroids with disjoint ground sets $E_1$ and
$E_2$. The collection 
\[
    \BB \ := \ \{ B_1\cup B_2 :  B_1\in \BB(\MM_1), B_2\in \BB(\MM_2)\}.
\]
is the set of bases of a matroid on $E_1 \cup E_2$, called the
\Defn{direct sum} of $\MM_1$ and $\MM_2$ and denoted by $\MM_1 \oplus \MM_2$.
The corresponding base configuration is exactly the Cartesian product
\begin{equation}\label{eqn:direct_sum}
    V_{\MM_1 \oplus \MM_2} \ = \  V_{\MM_1} \times V_{\MM_2}. 
\end{equation}    
If $E_1,\dots,E_r \subseteq E$ are the connected components of $\MM$, then
$\MM = \bigoplus_i \MM|_{E_i}$.  Since the dimension is additive with respect
to taking Cartesian products, showing that $\dim V_{\MM} = |E|-1$ if $\MM$ is
connected proves the following.

\begin{prop}\label{prop:V_dim}
    The smallest affine subspace containing $V_\MM$ is of dimension
    $|E|-c(\MM)$.
\end{prop} 

For a subset $X \subseteq E$ let us write $\ell_X(\x) = \sum_{e \in X} x_e$.
For $A \subseteq E$ we then have $\ell_X(\Char{A}) = |A \cap X|$. Hence
$\rk_{\MM}(X) = \max_{v \in P_\MM}\ell_X(v)$. For $X \subseteq E$ we define
the supporting hyperplane
\[
H_\MM(X) \ := \ \{ \x \in \R^E : \ell_X(\x) = \rk_\MM(X)\}.
\]
The corresponding faces of $V_\MM$ (or equivalently of $P_\MM$) are easy to describe.

\begin{prop}[\cite{Edmonds70}]\label{prop:X-faces}
    For a matroid $\MM = (E,\BB)$ and a subset $X \subset E$, we have
    \[
    V_\MM \cap H_\MM(X) \ = \ V_{\MM|_X \oplus \MM/X} \ = \ V_{\MM|_X} \times
    V_{\MM/X}.
    \]
\end{prop}

Let us illustrate this on our running example.

\begin{ex}[continued]
    The graph given in Example~\ref{ex:graph} yields a connected matroid on $4$ elements
    and hence a $3$-dimensional base configuration. The corresponding base
    polytopes is this:
    \begin{center}
        \includegraphics[scale=1]{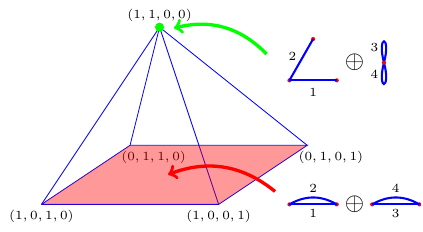}
    \end{center}
    The $5$ bases correspond to the vertices of $P_\MM$.  The set $X =
    \{3,4\}$ is of rank $1$ and the hyperplane corresponding to $x_3 + x_4 =
    1$ supports $P_\MM$ in the quadrilateral facet shown. As indicated, its
    vertices correspond exactly to the bases of $\MG|_{\{3,4\}}\times
    \MG/\{3,4  \}$. Likewise, the set $Y = \{1,2\}$ has rank $2$ and $x_1 +
    x_2 = 2$ supports $P_\MM$ in a vertex, which is the matroid base polytope
    of $\MG|_{\{1,2\}}\times \MG/\{1,2  \}$.
\end{ex}

We define the following families of matroids:
\begin{align*}
    \MLevels_{ k} &\ := \ \{ \MM \text{ matroid} : \Lev(V_\MM) \le k \},
    \text{and}\\
    \MTheta_{ k} &\ := \ \{ \MM \text{ matroid} : \Th(V_\MM) \le k \}.
\end{align*}
We will say that a matroid $\MM$ is of Theta rank or level $k$ if the
corresponding base configuration $V_\MM$ is. Now combining
Proposition~\ref{prop:X-faces} with Lemma~\ref{lem:face-hed} proves the main
theorem of this section.

\begin{thm}\label{thm:minor_closed}
    The classes $\MTheta_{ k}$ and $\MLevels_{ k}$ are closed under
    taking minors.
\end{thm}
\begin{proof}
    By definition, every minor $\NN$ of $\MM$ can be obtained by a sequence of
    restrictions and contractions. By repeatedly using
    Proposition~\ref{prop:X-faces}, we infer that $V_\NN$ is a face of $V_\MM$
    and Lemma~\ref{lem:face-hed} assures us that $\Th(V_\NN) \le \Th(V_\MM)$
    and $\Lev(V_\NN) \le \Lev(V_\MM)$.
\end{proof}

Let us analogously define the classes $\GTheta_{ k}$ and $\GLevels_{ k}$ of
graphic matroids of Theta rank and levelness bounded by $k$. These are also
closed under taking minors and the Robertson--Seymour's theorem
(\cite{RobertsonSeymour}) asserts that there is a finite list of excluded
minors characterizing each class.

In the remainder of the section we will recall the facet-defining hyperplanes
of $V_\MM$ which will also show that \emph{all} faces of $V_\MM$ correspond to
direct sums of minors. The facial structure of $V_\MM$ has been of interest
originally in combinatorial optimization~\cite{Edmonds70} (see
also~\cite[Ch.~40]{Schrijver}) and later in geometric
combinatorics and tropical geometry~\cite{Ardila,Sturmfels,Kim}. 

\begin{thm}\label{thm:matroidfacets}
    Let $\MM = (E,\BB)$ be a connected matroid. For every facet $U \subset
    V_\MM$ there is a unique $ \emptyset \neq S \subset E$ such that $U =
    V_\MM \cap H_\MM(S)$. Conversely, a subset $ \emptyset \neq S \subset E$
    gives rise to a facet if and only if one of the following conditions hold
    \begin{enumerate}[\rm (i)]
        \item $S$ is a flat such that $M|_S$ as well as $M/S$ are connected;
        \item $S = E \setminus e$ for some $e \in E$ such that $M|_S$ as well as $M/S$ are connected.
    \end{enumerate}
\end{thm}

In~\cite{Sturmfels} the subsets $S$ in (i) were called \Defn{flacets} and we
stick to this name. In our study of the Theta rank and the levelness of base
configurations, the following asserts that we will only need to consider
flacets. For brevity, a \Defn{$k$-level flacet} $F$ of a matroid $\MM$ refers
to flacet of $\MM$ whose facet-defining affine function $\ell_F(\x)$ for
$V_\MM$ is $k$-level.

\begin{prop}\label{prop:only_flacets}
    Let $\MM$ be a connected matroid and $S = E \setminus e$. Then
    $\ell_S(\x)$ takes $2$ values on $V_\MM$ and hence is $1$-sos.
\end{prop}
\begin{proof}
    For any basis $B \in \BB(\MM)$, we have $\ell_S(\1_B) = |S \cap B|$.
    Since every basis has the same cardinality, it follows that $\ell_S(\x)$
    takes only two distinct values.
\end{proof}

\begin{ex}
    The facets of the running example are four triangles and one square. The
    four triangles correspond to the two sets $\{1,2,4 \}$, $\{1,2,3\}$ of
    cardinality $|E|-1$ and the two flacets $\{2\}$, $\{1\}$, while the square
    corresponds to the flacet $\{3,4\}$. We have already described in the
    previous example the square facet. In the picture we highlight two
    triangular facets, the first one (green) corresponding to the flacet
    $\{1\}$, the second one (red) to the set $\{1,2,4\}$. The bases contained
    in the green triangle are exactly pairs of spanning trees in 
    shown deletion and contraction of $G$.
\begin{center}
\includegraphics[scale=1]{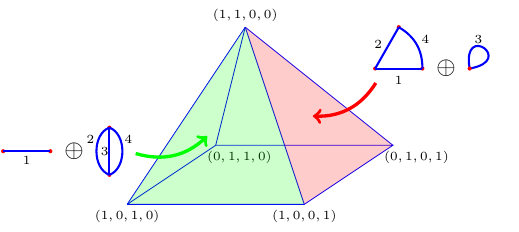}
\end{center}
\end{ex}

A seemingly trivial but useful class of matroids is given by the \Defn{uniform
matroids} $U_{n,k}$ for $0 \le k \le n$ given on ground set  $E = \{ 1,\dots,
n\}$ and bases $\BB(U_{n,k}) = \{ B \subseteq E : |B| = k \}$.

\begin{prop}\label{prop:TH-uniform}
    Uniform matroids are $2$-level and hence have Theta rank $1$.
\end{prop}
\begin{proof}
    The base polytope of $U_{n,k}$ is also known as the
    \emph{$(n,k)$-hypersimplex} and is given by
    \[
        P_{U_{n,k}} \ = \ \conv \{ \1_B : B \subseteq E, |B| = k\} \ = \ \Bigl\{ \x
        \in \R^E : 0 \le x_e \le 1, \sum_e x_e = k \Bigr\}.
    \]
    The facet-defining functions are among the functions $\{
    \pm \ell_{\{e\}}(\x) = \pm x_e : e \in E \}$ 
    which can take only two different values on $0/1$-points.
\end{proof}

\section{$2$-level matroids}
\label{sec:2level}

In this section we investigate the excluded minors for the class of $2$-level
matroids and, by Theorem~\ref{thm:GPT}, equivalently the matroids of Theta
rank $1$. In this case we can give the complete and in particular finite list
of forbidden minors. We start by showing that matroids with few elements and
of small rank are always $2$-level. By Proposition~\ref{prop:only_flacets} we
only need to inspect the levelness of flacets of a matroid.

\begin{prop}\label{prop:smallmatroids}
    Let $\MM=(E,\BB)$ be a matroid.
    If $\rk(\MM) \le 2$ or $|E| \le 5$, then $\MM$ is $2$-level.
\end{prop}
\begin{proof}
The case $\rk(\MM)=1$ is trivial since there is no proper flacet. On the other
hand, if $\rk(\MM)=2$ the proper flacets are necessarily flacets of rank $1$.
The linear function $\ell_F(\x)$ for any such flacet $F$ only takes values in
$\{0,1\}$ and thus is $2$-level.
By~\eqref{eqn:V-dual} and Proposition~\ref{prop:TH-aff}, $\MM$ and $\MM^*$
have the same Theta rank and levelness. If $|E| \le 5$, then either $\MM$ or
$\MM^*$ is of rank $\le 2$.
\end{proof}

A first example of a matroid of levelness $\ge 3$ is given by the graphic
matroid associated to the complete graph $K_4$.
\begin{prop}\label{prop:K4}
    The graphic matroid $\MM(K_4)$ is $3$-level.
\end{prop}
\begin{proof}
    Let $F = \{3,4,6\}$ be the flat corresponding to the labelled graph illustrated in Example \ref{ex:K4}.  Both the contraction of $F$ and the restriction to $F$ are
    connected (or biconnected on the level of graphs) and thus $F$ is a flacet
    with $\ell_F(\x) =  x_1+x_2+x_3$. The spanning trees $B_1 = \{1,2,5\}$ and
    $B_2 = \{1,5,6\}$ satisfy $|F \cap B_1| < |F \cap B_2| < \rk(F)$ which
    shows that $\MM(K_4)$ is at least $3$-level.  To see that $\MM(K_4)$ is at
    most $3$-level we notice that every proper flacet $F$ has rank smaller or
    equal than $ \rk(\MM(K_4)) - 1 =2$ and hence $\ell_F(\x)$ can take at most
    three different values.
\end{proof}

Before analyzing other matroids we quickly recall a \Defn{geometric
representation} of certain matroids of rank $3$: The idea is to draw a diagram
in the plane whose points correspond to the elements of the ground set. 
Any subset of $3$ elements constitute a basis unless they are contained in a
depicted line.

\begin{ex}\label{ex:K4}
    Let us consider the graph $K_4$ and its geometric representation as a
    matroid:
\begin{center}
\includegraphics[scale=0.6]{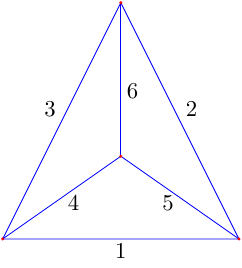} \qquad  \qquad \qquad
\includegraphics[scale=0.7]{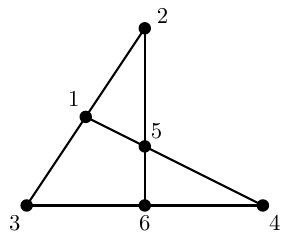}
\end{center}
Thus the geometric representation consists only of the four lines associated
to the $3$-circuits of $K_4$.
\end{ex}

Starting from the geometric representation of $\MM(K_4)$ we define three new matroids by removing one, two or three lines of the representation and we call them respectively $\mathcal{W}^3$, $Q_6$ and $P_6$. None of these matroids is graphic, but we can easily draw their geometric representations:
\begin{center}
\includegraphics[scale=0.6]{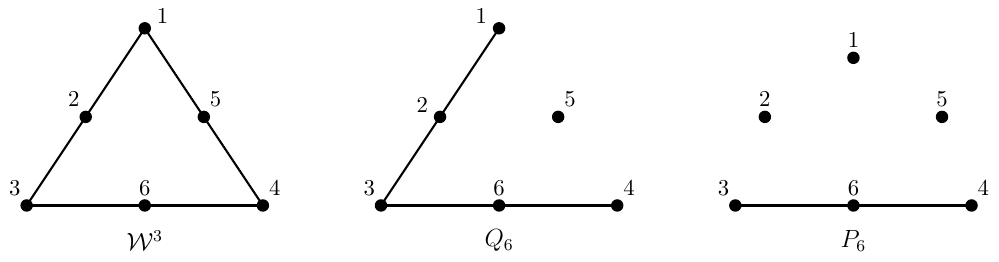}
\end{center}

\begin{prop}\label{prop:minor-list}
    The matroids $\mathcal{W}^3$, $Q_6$, and $P_6$ are $3$-level.
\end{prop}
\begin{proof}
    Let $\MM$ be any of the three given matroids and consider $F=\{ 3,4,6 \}$.
    It is easy to check that $\MM|_F \cong U_{3,2}$ and $\MM/F \cong U_{3,1}$
    which marks $F$ as a flacet. The vertices of the matroid polytope
    associated to the bases $\{1,2,5\},\{1,2,4\},\{ 1,3,4 \}$ lie on distinct
    hyperplanes parallel to $H_\MM(F) = \{ \ell_F(\x) = \rk_\MM(F) \}$.
    Therefore the matroids are at least $3$-level. Since $\rk(\MM)=3$, we can
    use the same argument as in the proof of Proposition~\ref{prop:K4}.
\end{proof}

The list of excluded minors for $\MLevels_{2}$ so far includes $\MM(K_4)$,
$\mathcal{W}^3$, $Q_6$, and $P_6$. To show that this list is complete, we will
approach the problem from the constructive side and consider how to
synthesize $2$-level matroids. We already saw that $\MLevels_{2}$ is closed
under taking direct sums. We will now consider three more operations that
retain levelness. Let $\MM_1 = (E_1, \BB_1)$ and $\MM_2 = (E_2, \BB_2)$ be
matroids such that $\{p\} = E_1 \cap E_2$.  We call $p$ a \Defn{base point}.
If $p$ is not a coloop of both, then we define the \Defn{series
connection} $\SeriesConn(\MM_1,\MM_2)$ with respect to $p$ as the matroid on
ground set $E_1 \cup E_2$ and with bases
\[
    \BB \ = \ \{ B_1 \cup B_2 : B_1 \in \BB_1, B_2 \in \BB_2, B_1 \cap B_2 =
    \emptyset \}.
\]
We also define the \Defn{parallel connection} with respect to $p$ as the
matroid $\SeriesConn(\MM_1^*,\MM_2^*)^*$ provided $p$ is not a loop of both.
Notice that $\SeriesConn(\MM_1,\MM_2)$ contains both $\MM_1$ and $\MM_2$ as a
minor.  

The operations of series and parallel connection, introduced by
Brylawski~\cite{Brylawski}, are inspired by the well-known series and parallel
operations on graphs. The following example illustrates the construction in
the graphic case.

\begin{ex}
Let us consider again the two graphic matroids $U_{3,2}$ and $ \MM(K_4)$. Their series connection is the following graph:
\begin{center}
 \includegraphics[scale=0.7]{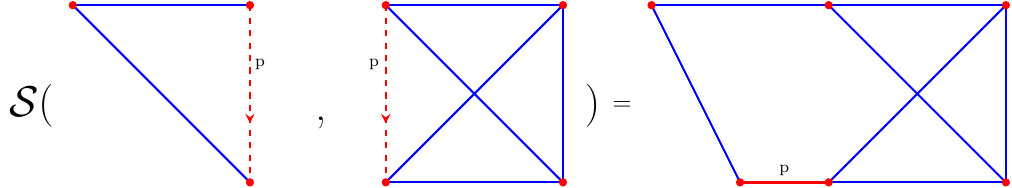}
 \end{center}
\end{ex}

An extensive treatment of these two operations is given
in~\cite[Sect.~7.1]{Oxley}. We focus here on the geometric properties from
which many combinatorial consequences can be deduced. Since $E_1 \cap E_2 =
\{p\}$, we will write $E_1 \uplus E_2 = (E_1 \cup E_2  \cup \{p_1,p_2\})
\setminus \{p\}$ for the \emph{disjoint union} in the following result. Thus,
$p_1$ and $p_2$ corresponds to $p \in E_1$ and $p \in E_2$, respectively.

\begin{lemma}\label{lem:P_series}
    Let $\MM_1 = (E_1,\BB_1)$ and $\MM_2 = (E_2,\BB_2)$  be matroids with
    $\{p\} = E_1 \cap E_2$ not a coloop of both.
    Then the base polytope $P_\SeriesConn$ of the series connection
    $\SeriesConn = \SeriesConn(\MM_1,\MM_2)$ is linearly isomorphic to
    \[
    (P_{\MM_1} \times P_{\MM_2}) \cap \{ \x \in \R^{E_1\uplus E_2} : x_{p_1} +
    x_{p_2} \le 1 \}.
    \]
\end{lemma}
\begin{proof}
    It is clear that the base configuration $V_\SeriesConn$ is isomorphic
    to
    \[
    V^\prime \ = \ (V_{\MM_1} \times V_{\MM_2}) \cap \{ \x \in \R^{E_1\uplus E_2} : x_{p_1} +
    x_{p_2} \le 1 \}
    \]
    under the linear map $\pi : \R^{E_1 \uplus E_2} \rightarrow \R^{E_1 \cup
    E_2}$ given by $\pi(\1_{p_1}) =  \pi(\1_{p_2}) = \1_{p}$ and $\pi(\1_{e})
    = \1_e$ otherwise. Indeed, let $r_i = \rk(\MM_i)$, then a linear inverse
    is given by $s : \R^{E_1 \cup E_2} \rightarrow \R^{E_1 \uplus E_2}$ with
    $s(\x)_{p_i} = r_i - \ell_{E_i}(\x)$ for $i=1,2$ and the identity
    otherwise. 
    
    It is therefore sufficient to show that the vertices of
    \[
        P^\prime \ = \ (P_{\MM_1} \times P_{\MM_2}) \cap \{ \x \in
        \R^{E_1\uplus E_2} : x_{p_1} + x_{p_2} \le 1 \}.
    \]
    are exactly the points in $V^\prime$. Clearly $V^\prime$ is a subset of
    the vertices and any additional vertex of $P^\prime$ would be the
    intersection of the relative interior of an edge of $P_{\MM_1} \times
    P_{\MM_2}$ with the hyperplane $H = \{ \x : x_{p_1} + x_{p_2} = 1 \}$. However,
    every edge of $P_{\MM_1} \times P_{\MM_2}$ is
    parallel to some $\1_e - \1_f$ for $e,f \in E_1$ or $e,f \in E_2$. Thus
    every edge of $P_{\MM_1} \times P_{\MM_2}$ can meet $H$ only in one of its
    endpoints which proves the claim.
\end{proof}

It is interesting to note that the operation that
related $P_{\MM_1}$ and $P_{\MM_2}$ to $P_{\SeriesConn(\MM_1,\MM_2)}$ is
exactly a \emph{subdirect product} in the sense of McMullen~\cite{mcmullen}.
From the description of $P_{\SeriesConn(\MM_1,\MM_2)}$ we instantly get
information about the Theta rank and levelness of the series and parallel
connection.

\begin{cor}\label{cor:Series_level}
    Let $\SeriesConn = \SeriesConn(\MM_1,\MM_2)$ be the series connection of
    matroids $\MM_1$ and $\MM_2$. Then
    \[
        \Th(\SeriesConn) \ = \ \max( \Th(\MM_1), \Th(\MM_2) ).
    \]
    The same holds true for the parallel connection as well as for the levelness.
\end{cor}
\begin{proof}
    Lemma~\ref{lem:P_series} shows that the facet-defining affine functions of
    $P_\SeriesConn$ are among those of ${P_{\MM_1} \times P_{\MM_2}}$ and
    $\ell(\x) = x_{p_1} + x_{p_2}$. However, by the characterization of the
    bases of $\SeriesConn$, $\ell(\x)$ can take only values in $\{0,1\}$.
    Hence, $\Th(V_\SeriesConn) \ = \ \Th(V_{\MM_1}\times V_{\MM_2})$ and
    Proposition~\ref{prop:TH-product} finishes the proof.
\end{proof}

\begin{cor}\label{cor:Series_closed}
    The classes $\MTheta_{ k}$ and $\MLevels_{ k}$ are closed under
    taking series and parallel connections.
\end{cor}

The most important operation that we will need is derived from the series
connection. Let $\MM_1 = (E_1,\BB_1)$ and $\MM_2 = (E_2,\BB_2)$ be matroids
with $E_1 \cap E_2 = \{p\}$. If $p$ is not a loop nor a coloop for neither $\MM_1$ nor $\MM_2$, then we define
the \Defn{$2$-sum} 
\[
    \MM_1 \oplus_2 \MM_2 \ := \ \SeriesConn(\MM_1,\MM_2) / p.
\]
This is the matroid on the ground set $E = (E_1 \cup E_2) \setminus p$ and
with bases
\[
    \BB \ := \ \{ B_1\cup B_2\setminus p  :  B_1\in \BB_1, B_2\in
    \BB_2, p\in B_1 \triangle B_2 \}
\]
where $B_1 \triangle B_2$ is the symmetric difference.

The 2-sum is an associative operation for matroids which defines, by analogy
to the direct sum, the 3-connectedness: a connected matroid $\MM$ is
\Defn{3-connected} if and only if it cannot be written as a 2-sum of two
matroids each with fewer elements than $\MM$.

\begin{ex}
Let us consider the $2$-sum of a matroid $U_{3,2}\bigoplus_2 \MM(K_4)$: both matroids are graphic, therefore we can illustrate the operation for the corresponding graphs.
\begin{center}
 \includegraphics[scale=0.7]{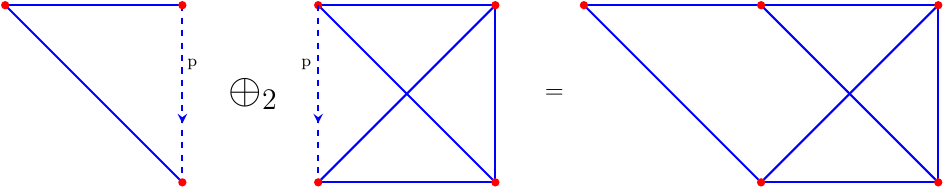}
 \end{center}
To perform the $2$-sum we select an element for each matroid, while in the picture it looks like we also need to orient the chosen element. This is the case only because we are drawing an embedding of a graphic matroids; in fact the structure given by the vertices is forgotten when we look at the matroid. Whitney's $2$-Isomorphism Theorem \cite[Thm.~5.3.1]{Oxley} clarifies that the matroid structure does not depend on the orientation we decide for the chosen elements.
\end{ex}

We will need the following two properties of $2$-sums.

\begin{lemma}[{\cite[Lem.~2.3]{Chaourar}}]\label{lem:3-uniform}
Let $\MM$ be a $3$-connected matroid having no minor isomorphic to any of
$\MM(K_4)$, $\mathcal{W}^{3}$, $Q_6$, $P_6$. Then $\MM$ is uniform.
\end{lemma}
\begin{lemma}[{\cite[Thm.~8.3.1]{Oxley}}]\label{lem:not3con}
Every matroid that is not $3$-connected can be constructed from $3$-connected proper minors of itself by a sequence of direct sums and $2$-sums.
\end{lemma}

We can finally give a complete characterization of the class $\MLevels_{2}
= \MTheta_{ 1}$.

\begin{thm}\label{thm:main1}
    For a matroid $\MM$ the following are equivalent.
    \begin{enumerate}[\rm (i)]
        \item $\MM$ has Theta rank $1$.
        \item $\MM$ is $2$-level.
        \item $\MM$ has no minor isomorphic to $\MM(K_4)$, $\mathcal{W}^3$,
            $Q_6$, or $P_6$.
        \item $\MM$ can be constructed from uniform matroids by taking direct
            or $2$-sums.
    \end{enumerate}
\end{thm}
\begin{proof}
    (i) $\Rightarrow$ (ii) is just Theorem~\ref{thm:GPT}. (ii) $\Rightarrow$
    (iii) follows from Theorem~\ref{thm:minor_closed} and
    Proposition~\ref{prop:minor-list}. Let $\MM$ be a matroid satisfying
    (iii). If $\MM$ is $3$-connected, then $\MM$ is uniform by
    Lemma~\ref{lem:3-uniform}. If $\MM$ is not $3$-connected, then, by
    Lemma~\ref{lem:not3con}, we can decompose $\MM$ into $3$-connected
    matroids and we may repeat the argument for each of these matroids. This
    shows (iv).  Finally, uniform matroids have Theta rank $1$ by
    Proposition~\ref{prop:TH-uniform}. Theta rank $\le k$ is retained by
    series connection (Corollary~\ref{cor:Series_level}) and, by definition,
    also by the $2$-sum.
\end{proof}

\begin{ex}\label{ex:serpargraphs}
    If we look at the family of $2$-level graphic matroids, the only excluded
    minor is the graph $K_4$. The class of graphs which do not contain $K_4$
    as a minor is the well-known class of \Defn{series-parallel graphs}
    $\GSeriesPar$. The theorem implies $\GLevels_{ 2}=\GSeriesPar$.
\end{ex}

There are other point configurations that are naturally associated to a
matroid $\MM$, most notably the configuration $D_\MM = \{ \Char{X} : X
\subseteq E \text{ dependent} \}$. For \emph{binary} matroids, the associated
polytope (up to translation and scaling) is called the \emph{cycle polytope}.
The practical relevance stems from the situation where $\MM = \MM(G)^*$ for
some graph $G$.  In this case, $D_\MM$ represents the collection of cuts in
$G$ which are important in combinatorial optimization.  The Theta rank of
$D_\MM$ has been studied in~\cite{GLPT}.  In particular, the paper gives a
characterization of binary matroids with $\Th(D_\MM)=1$ in terms of forbidden
minors with some additional conditions on the cocircuits.  The situation is
slightly different as the Theta rank of circuit configurations is monotone
with respect to deletion minors but not necessarily with respect to
contraction minors.  The characterization of $2$-level cut polytopes has been
also obtained by Sullivant~\cite{Sullivant}.

\section{Generation and psd rank}
\label{sect:gen_psd}

In this section we study two further face-hereditary properties of point
configurations that are intimately related to Theta-$1$ configurations.

\subsection{Degree of generation} 

For a point configuration $V \subset \R^d$, the \Defn{vanishing ideal} of $V$
is 
\[
    I(V) \ := \ \{ f(\x) \in \R[\x] = \R[x_1,\dots,x_d] : f(v) = 0 \text{ for
    all } v \in V \}.
\]

We say that $V$ is \Defn{generated in degree $\leq k$} if $I(V)$ is generated
as an ideal by $\{ f \in I(V) : \deg f \le k \}$ and we write $\Gen(V) = k$ if
$k$ is minimal with that property.  We define
\[
    \VGen_{ k} \ := \ \{ V \text{ point configuration} : \Gen(V) \le k \}.
\]

It is clear that $\Gen(V)$ is an affine invariant and, since all point
configurations are finite, we get
\begin{prop}
    The class $\VGen_{ k}$ is face-hereditary.
\end{prop}
\begin{proof} 
    Let $H = \{ p : \ell(p) = 0 \}$ be a supporting hyperplane for $V$.  Since
    $V$ is finite, the vanishing ideal of $V^\prime = V \cap H$ is the ideal
    generated by $I(V)$ and $\ell(\x)$. Since $\ell(\x)$ is affine, this then
    shows that $\Gen(V^\prime) \le \Gen(V)$. 
\end{proof}

The relation to point configurations of Theta rank $1$ is given by the
following proposition which is implicit in~\cite{GPT}.

\begin{prop}\label{prop:gen2}
    If $V \subset \R^d$ is a point configuration of Theta rank $1$, then
    $\Gen(V) \le 2$.
\end{prop}
\begin{proof}
    From Theorem~\ref{thm:GPT} we infer that the points $V$ are in convex
    position and the polytope $P=\conv(V)$ is $2$-level. We may assume that
    the configuration is spanning and hence up to affine equivalence, the
    polytope is given by
    \[
        P \ = \ \left\{ p \in \R^d : 
            \begin{array}{r@{\ }c@{\ }l@{\ }l} 
                0 \le & p_i & \le  1 &\text{ for } i = 1,\dots,d \\
                \delta_j^-  \le & \ell_j(p) & \le \delta_j^+ &\text{ for }
                j=1,\dots,n
            \end{array}
        \right\}
    \]
    for unique linear functions $\ell_j(\x)$ and $\delta_j^- < \delta_j^+$. In
    particular, $V \subseteq \{0,1\}^d$.  We claim that $I(V)$ is generated by
    the quadrics 
    \[
        x_i(x_i-1)  \quad \text{for } 1 \le i \le d, 
        \qquad 
        (\ell_j(\x) - \delta_j^-)(\ell_j(\x)-\delta^+_j)  \quad \text{for } 1
            \le j \le n.
    \]
    The vanishing locus $U$ is a smooth subset of $\{0,1\}^d$.  Thus, the
    polynomials span a real radical ideal. Now, every vertex $v \in V
    \subseteq \{0,1\}^d$ satisfies $\ell_j(v) = \delta_j^{\pm}$. Hence $V
    \subseteq U$. Conversely, every $u \in U$ is a vertex of $P$ and hence $U
    \subseteq V$.
\end{proof}

The following example illustrates the fact that degree of generation is
invariant under projective transformations while Theta rank is not.
\begin{ex}\label{ex:gen2}
To see that generation in degrees $\le 2$ is necessary for Theta rank $1$ but not
sufficient, consider the planar point configuration $V = \{
(1,0),(0,1),(2,0),(0,2)\}$. The configuration is clearly not $2$-level and
hence not Theta $1$, however the vanishing ideal $I(V)$ is generated by $x_1x_2$ and
$(x_1+x_2-1)(x_1+x_2-2)$ which implies $\Gen(V)\le 2$.
\begin{center}
\includegraphics[scale=0.7]{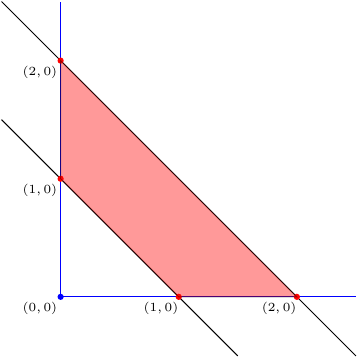}
\end{center}
\end{ex}

The vanishing ideals of base configurations are easy to write down explicitly.

\begin{prop}\label{prop:matroid_ideal}
    Let $\MM = (E,\BB)$ be a matroid of rank $r$. The vanishing ideal for
    $V_\MM$ is generated by
    \[
        x_e^2-x_e \text{ for all } e \in E, \quad \ell_E(\x) -r, \quad
        \x^C=\prod_{e\in C}x_e \text{ for all circuits } C \subset E.
        \]
\end{prop}
\begin{proof}
    Any complex solution to the first two sets of equations is of the form
    $\1_B \in \{0,1\}^E$ for some $B \subseteq E$ with $|B| = r$ and, in
    particular, is a real point. For the last set of equations, we note that
    $(\1_B)^C = 0$ for all circuits $C$ if and only if $B$ does not contain a
    circuit. This is equivalent to $B \in \BB$. To show that the ideal is real
    radical, we can argue as in the proof of Proposition~\ref{prop:gen2}: Each
    of the finitely many points is smooth.
\end{proof}

Let us write $\MGen_{ k}$ for the class of matroids $\MM$ with
$\Gen(V_\MM)\le k$. The previous proposition is a little deceiving in the
sense that it suggests a direct connection between the size of circuits and
the degree of generation. This is not quite true. Indeed, let $G = K_4
\setminus e$ be the complete graph on $4$ vertices minus an edge. Then
$\MM(G)$ has a circuit of cardinality $4$ but $\MM(G) \in \MLevels_{ 2}
\subseteq \MGen_{ 2}$ by Theorem~\ref{thm:main1} and
Proposition~\ref{prop:gen2}. The main result of this section is that for base
configurations the condition of Proposition~\ref{prop:gen2} is also
sufficient.

\begin{thm}\label{thm:main2}
    Let $\MM$ be a matroid. Then $V_\MM$ is Theta $1$ if and only if
    $\Gen(V_\MM) \le 2$.
\end{thm}

\begin{proof}
From Proposition~\ref{prop:gen2} we already know that $\MTheta_{1} \subseteq
\MGen_{2}$. Now, if $\MM \in \MGen_{ 2} \setminus \MTheta_{ 1}$, then $\MM$
has a minor isomorphic to $\MM(K_4)$, $P_6$, $Q_6$, or $\mathcal{W}^3$. Since
$\MGen_{ 2}$ is closed under taking minors, the following proposition yields a
contradiction.
\end{proof}

\begin{prop}
$\MM(K_4)$, $\mathcal{W}^3$, $Q_6$, and $P_6$ are not in $\MGen_{ 2}$.
\end{prop}
\begin{proof}
    For a point configuration $V \subset \R^n$, let
    $I \subset \R[x_1,\dots,x_n]$ be its vanishing ideal. If $I$ is generated
    in degrees $\le k$, then so is any Gr\"obner basis of $I$ with respect to
    a degree-compatible term order. The claim can now be verified by, for
    example, using the software \emph{Macaulay2}~\cite{M2}.
\end{proof}

\newcommand\PSD{\mathcal{S}}
\subsection*{Psd rank and minimality}
Let $\PSD^m \subset \R^{m \times m}$ be the vector space of symmetric $m \times
m$ matrices. The \Defn{psd cone} is the closed convex cone
$
\PSD^m_+ \ = \ \{ A \in \PSD^m : A \text{ positive semidefinite} \}.
$

\begin{defi}
    A polytope $P \subset \R^d$ has a \Defn{psd-lift} of size $m$ if there is
    an affine subspace $L \subset \PSD^m$ and a linear projection $\pi : \PSD^m
    \rightarrow \R^d$
    such that $  P = \pi(\PSD^m_+ \cap L)$.
     The \Defn{psd rank} $\Psd(P)$ is the size of a smallest psd-lift.
\end{defi} 

Psd-lifts together with lifts for more general cones were introduced by
Gouveia, Parrilo, and Thomas~\cite{GPT2} as natural generalization of
\emph{polyhedral lifts} or \emph{extended formulations}. Let us define
$\VPsd_{ k}$ as the class of point configurations $V$ in convex position
such that $\conv(V)$ has a psd-lift of size $\le k$.  In~\cite{GRT} it was
shown that for a $d$-dimensional polytope $P$ the psd rank is always $\ge
d+1$. A polytope $P$ is called \Defn{psd-minimal} if $\Psd(P) = \dim P +1$. We
write $\VPmin$ for the class of psd-minimal (convex position) point
configurations.

\begin{prop}\label{prop:PSD_closed}
    The classes $\VPsd_{ k}$ and $\VPmin$ are face-hereditary.
\end{prop}
\begin{proof}
    Let $V \in \VPsd_{ k}$ and let $(L,\pi)$ be a psd-lift of $P=\conv(V)$.
    For a supporting hyperplane $H$ we observe that $(L \cap \pi^{-1}(H),
    \pi)$ is a psd-lift of $P \cap H$ of size $m\leq k$.

    Let $P$ be psd-minimal and let $F = P \cap H$ a face of dimension $\dim F
    = \dim P -1$. If $F$ is not psd-minimal, then by~\cite[Prop.~3.8]{GRT},
    $\Psd(P) \ge \Psd(F)+1 > \dim F + 2 = \dim P + 1$.
\end{proof}

A characterization of psd-minimal polytopes in small
dimensions was obtained in~\cite{GRT} and, in particular, the following
relation was shown.
\begin{prop}[{\cite[Cor.~4.2]{GRT}}]\label{prop:psd}
    Let $V$ be a point configuration in convex
    position. If $\Th(V) = 1$, then $P = \conv(V)$ is psd-minimal.
\end{prop}

In~\cite{GRT} an example of a psd-minimal polytope that is not $2$-level is
given, showing that the condition above is sufficient but not necessary.
The main result of this section is that the situation is much better for base
configurations.

\begin{thm}\label{thm:main3}
    Let $\MM$ be a matroid. The base polytope $P_\MM = \conv(V_\MM)$ is
    psd-minimal if and only if $\Th(\MM) = 1$.
\end{thm}

In light of Proposition~\ref{prop:psd} it remains to show that there is no
psd-minimal matroid $\MM$ with $\Th(\MM) > 1$. Since $\VPmin$ is
face-hereditary, it is sufficient to show that the excluded minors
$\MM(K_4)$, $\mathcal{W}^3$, $Q_6$, and $P_6$ are not psd-minimal. 

In order to do so, we need to recall the connection to slack
matrices and Hadamard square roots developed in~\cite{GRT}. For a more
coherent picture of the relations in particular to cone factorizations we
refer to the papers~\cite{GPT2,GRT}.
Let $P$ be be a
polytope with vertices $v_1,\dots,v_t$ and facet-defining affine functions
$\ell_j(\x) = \beta - \langle a_j, \x\rangle$ for $j = 1,\dots,f$.
The \Defn{slack matrix} of $P$ is the non-negative matrix
$S_P\in\mathbb{R}^{t\times f}$ with
\[
    (S_P)_{ij}=\beta_j-\langle a_j,v_i\rangle
\]
for $i=1,\dots,t$ and $j=1,\dots,f$. A \Defn{Hadamard square root} of $S_P$ is
a matrix $H \in \R^{t \times f}$ such that $(S_P)_{ij} = H^2_{ij}$ for all
$i,j$. Moreover, we define $\Hrk S_P$ as the smallest rank among all
Hadamard square roots. The following is the main connection between Hadamard
square roots and the psd-rank.

\begin{thm}[{\cite[Thm.~3.5]{GRT}}]\label{thm:hsrpsdmin}
    A polytope $P$ is psd-minimal if and only if \[
    \Hrk(S_P) \ = \  \dim P + 1.\]
\end{thm}

The matroid base polytopes of $\MM(K_4),\mathcal{W}^3, Q_6, P_6$ are all of
dimension $5$. Thus, to complete the proof of Theorem~\ref{thm:main3}, it
suffices to show that for each of these four matroids, the Hadamard square
roots of the corresponding slack matrices have rank at least $7$. Using
Proposition~\ref{prop:PSD_closed}, this then implies that every matroid that
has a minor isomorphic to $\MM(K_4),\mathcal{W}^3, Q_6,$ or $P_6$ is not
psd-minimal.  We start with a technical result.

\begin{prop}\label{prop:tech}
    The matrix
\[
    A_0 \ =  \ \begin{pmatrix}
        0 & 1 & 1 & 1 \\
        1 & 0& 1 & 1 \\
        1 & 1 & 0 & 1 \\
        1 & 1 & 1 & 0 
        \end{pmatrix}
\]
has $\Hrk A_0 = 4$.
\end{prop}
\begin{proof}
    Every Hadamard square root of $A_0$ is of the form
    \[
        H \  = \ \begin{pmatrix}
            0 & y_1 & y_2 & y_3 \\
            y_4 & 0 & y_5 & y_6 \\
            y_7 & y_8 & 0 & y_9 \\
            y_{10} & y_{11} & y_{12} & 0 
        \end{pmatrix}
        \]
    with $y_i^2=1$, $i=1,..,12$.  Claiming that $\Hrk A_0 = 4$ is
    equivalent to the claim that every Hadamard square root $H$ is
    non-singular. Using the computer algebra software
    \emph{Macaulay2}~\cite{M2} it can be checked that the ideal
    \[
        I \ =  \ \langle y_1^2-1,...,y_{12}^2-1, \det H \rangle \ \subseteq \
        \mathbb{C}[y_1,\dots,y_{12}]
    \]
    contains $1$ which excludes the existence of a rank-deficient Hadamard
    square root.
\end{proof}

\begin{prop}
    Let $P = P_\MM$ the base polytope for $\MM \in \{
    \MM(K_4), \mathcal{W}^3, Q_6, P_6\}$. Then $\Hrk (S_P) \geq
    7$.
\end{prop}
\begin{proof}
      We explicitly give the argument for $\MM = \MM(K_4)$ and $P = P_\MM$. This proof works also for the other matroids for the same choice of the collection of bases and flacets. It will be sufficient to find a $7 \times 7$-submatrix $A$ of
    $S_P$ with $\Hrk(N) \ge 7$. Consider the following collection of bases and
    flacets of $\MM$:
\[
        \begin{array}{l@{ \ \ = \ \ }l@{\hspace{1cm}}l@{ \ \ = \ \ }l}
            B_1 & \{1,2,4\} & F_1 & \{ 1 \}\\
            B_2 & \{1,2,5\} & F_2 & \{ 2 \}\\
            B_3 & \{1,2,6\} & F_3 & \{ 3 \}\\
            B_4 & \{1,3,6\} & F_4 & \{ 4 \}\\
            B_5 & \{1,4,6\} & F_5 & \{ 5 \}\\
            B_6 & \{1,5,6\} & F_6 & \{ 6 \}\\
            B_7 & \{2,4,6\} & F_7 & \{ 3,4,6 \}\\
        \end{array}
    \]
   It is straightforward to verify by means of Theorem \ref{thm:matroidfacets}
   that the subsets $F_1, \ldots, F_7$ are flacets of $\MM$. Notice that the
   list is not complete, but to prove the claim, it suffices to find a
   suitable submatrix of the complete slack matrix. Consider the induced
   submatrix of $S_P$
    \[
        A \ = \ 
        \bordermatrix{ %
        ~& {\scriptstyle \{1\}} & {\scriptstyle \{2\}} & {\scriptstyle \{3\}} &
           {\scriptstyle \{4\}} & {\scriptstyle \{5\}} & {\scriptstyle \{6\}} &
           {\scriptstyle \{3,4,6\}}\cr %
            {\scriptstyle \{1,2,4\}}& 0 & 0 & 1 & 0 & 1 & 1 & 1  \cr
            {\scriptstyle \{1,2,5\}}& 0 & 0 & 1 & 1 & 0 & 1 & 2  \cr
            {\scriptstyle \{1,2,6\}}& 0 & 0 & 1 & 1 & 1 & 0 & 1  \cr
            {\scriptstyle \{1,3,6\}}& 0 & 1 & 0 & 1 & 1 & 0 & 0  \cr
            {\scriptstyle \{1,4,6\}}& 0 & 1 & 1 & 0 & 1 & 0 & 0  \cr
            {\scriptstyle \{1,5,6\}}& 0 & 1 & 1 & 1 & 0 & 0 & 1  \cr
            {\scriptstyle \{2,4,6\}}& 1 & 0 & 1 & 0 & 1 & 0 & 0  \cr
        }
    \]
Then $\Hrk(A) = 7$ if and only if 
\newcommand\HL[1]{\fcolorbox{red}{white}{$#1$}}
\[
\begin{vmatrix}
        0 & 0 & \pm 1 & 0 & \pm 1 & \pm 1 & \pm 1  \\
        0 & 0 & \pm 1 & \pm 1 & 0 & \pm 1 & \pm \sqrt{2}  \\
        0 & 0 & \pm 1 & \pm 1 & \pm 1 & 0 & \pm 1  \\
        0 & \pm 1 & 0 & \pm 1 & \pm 1 & 0 & 0  \\
        0 & \pm 1 & \pm 1 & 0 & \pm 1 & 0 & 0  \\
        0 & \pm 1 & \pm 1 & \pm 1 & 0 & 0 & \pm 1  \\
        \fcolorbox{red}{white}{$\pm 1$} & 0 & \pm 1 & 0 & \pm 1 & 0 & 0  \\
        \end{vmatrix} \ = \
    \pm \begin{vmatrix}
        0 & \pm 1 & 0 & \pm 1 & \pm 1 & \pm 1  \\
        0 & \pm 1 & \pm 1 & 0 & \pm 1 & \HL{\pm\sqrt{2}}  \\
        0 & \pm 1 & \pm 1 & \pm 1 & 0 & \pm 1  \\
        \pm 1 & 0 & \pm 1 & \pm 1 & 0 & 0  \\
        \pm 1 & \pm 1 & 0 & \pm 1 & 0 & 0  \\
        \pm 1 & \pm 1 & \pm 1 & 0 & 0 & \pm 1 
    \end{vmatrix} \ \neq \  0 .
\]
The last determinant is of the form $a + \sqrt{2} \cdot b$ for some integers
$a,b$. To check that this determinant is nonzero, we can check that $b$ is
nonzero.  By Laplace expansion, this is the case if
\[
    \begin{vmatrix}
        0 & \pm 1 & 0 & \pm 1 &  \HL{\pm 1}  \\
        0 & \pm 1 & \pm 1 & \pm 1 & 0  \\
        \pm 1 & 0 & \pm 1 & \pm 1 & 0 \\
        \pm 1 & \pm 1 & 0 & \pm 1 & 0 \\
        \pm 1 & \pm 1 & \pm 1 & 0 & 0 
    \end{vmatrix} \ = \ 
    \pm \begin{vmatrix}
        0 & \pm 1 & \pm 1 & \pm 1  \\
        \pm 1 & 0 & \pm 1 & \pm 1 \\
        \pm 1 & \pm 1 & 0 & \pm 1 \\
        \pm 1 &  \pm 1 & \pm 1 & 0 
        \end{vmatrix} \ \neq  \ 0.
\]
The latter is exactly the claim that the matrix $A_0$ of 
Proposition~\ref{prop:tech} has $\Hrk(A_0)=4$.
\end{proof}

\section{Higher level graphs}\label{sec:higher}

In this section we study the class $\GLevels_k$ of $k$-level graphic matroids
for arbitrary $k$.  For sake of brevity, we will simply refer to them as
`graphs'. The Robertson-Seymour theorem assures that the list of forbidden
minors characterizing $\GLevels_k$ is finite and we give an explicit
description in the next subsection.  In Section~\ref{ssec:3level}, we focus on
the class of $3$-level graphs which is characterized by exactly one forbidden
minor, the wheel $W_4$ with $4$ spokes. The class of $W_4$-minor-free graphs
was studied by Halin and we recover its building blocks from levelness
considerations. In Section~\ref{ssec:4level} we focus on the class of graphs
with Theta rank $2$. Forbidden minors for this class can be obtained from the
structure of $4$-level graphs.

\subsection{Excluded minors for $k$-level graphs}

A consequence of Theorem~\ref{thm:main1} is that a graph $G$ is $2$-level if
and only if $G$ does not have $K_4$ as a minor. In order to give a 
characterization of $k$-level graphs in terms of forbidden minors, we first
need to view $K_4$ from a different angle. 
\begin{defi}\label{dfn:cone}
    The \Defn{cone} over a graph $G = (V,E)$ with apex $w \not\in V$ is the
    graph 
    \[
        \cone(G) = (V \cup \{w\},E \cup \{ wv : v \in V\} ).
    \]
\end{defi}

Let us denote by $C_n$ the \Defn{$n$-cycle}. Thus, we can view $K_4$ as the
cone over $C_3$. As in the previous section, we only need to consider graphic
matroids $\MM(G)$ which are connected. In terms of graph theory these correspond
exactly to biconnected graphs.
For a flacet $F$ let us denote by $V_F \subseteq V$ the vertices covered by
$F$.
\begin{prop}\label{prop:V_closed}
    Let $G = (V,E)$ be a biconnected graph and $F \subset E$ a flacet with $|E
    \backslash F| \ge 2$. Then $G|_F$ is a vertex-induced subgraph.
\end{prop}
\begin{proof}
    By contradiction, suppose that $e\in E \backslash F$ is an edge with both endpoints
    in $V_F$. Since $F$ is a flacet, $G/F$ is a biconnected graph with loop
    $e$. This contradicts $|E \backslash F| \ge 2$.
\end{proof}

The definition of flacets requires the graph $G/F$ to be biconnected.
This, in turn, implies that $G|_{E\backslash F}$ is connected. Let us write $C(F): = \{
uv \in E: u \in V_F, v \not\in V_F\}$ for the induced \Defn{cut}.  Moreover,
let us write $\oF := E \setminus (F \cup C(F))$. The next result allows us to
find minors $G'$ of $G$ with $\Lev(G') = \Lev(G)$.

\begin{lemma}\label{lem:contract_oF}
    Let $G$ be a biconnected graph and $F$ a $k$-level flacet of $\MM(G)$.
    Then $F$ is a $k$-level flacet of $\MM(G/\oF)$.
\end{lemma}
\begin{proof}
    Let $H = G/\overline{F}$. It follows from the definition of flacets, that
    $G|_{\oF}$ is connected and thus $H/F = G/(F \cup \oF)$ is biconnected,
    since $\MM(H/F)= U_{|C(F)|,1}$ is a connected matroid. Moreover $H|_F =
    G|_F$ is biconnected and therefore $F$ is a flacet of $H$. 

    For the levelness of $F$, observe that it cannot be bigger than $k$. Let
    $T_1 \subset E$ be a spanning tree such that the restriction to the
    connected graph $G|_{E \setminus F}$ is also a spanning tree. In
    particular, $|T_1 \cap F|$ is minimal among all spanning trees. It now
    suffices to show that there is a sequence of spanning trees
    $T_1,T_2,\dots,T_k \subset E$ with $|T_i \cap F| = |T_1 \cap F| + i-1$ for
    all $i=1,\dots,k$ and such that $T_i \cap \oF = T_j \cap \oF$ for all
    $i,j$.  The contractions $T_i / \oF$ then show that $F$ is at least
    $k$-level for $H$.

    If $T_i \cap F$ is not a spanning tree for $G|_F$, then pick $e \in F
    \setminus T_i$ such that $e$ connects two connected components of
    $(V_F,T_i \cap F)$.  Since $T_i$ is a spanning tree, there is a cycle in
    $T_i \cup e$ that uses at least one cut edge $f \in C(F) \cap T_i$. Hence
    $T_{i+1} = (T_i \setminus e) \cup f$ is the new spanning tree with the
    desired properties.
\end{proof}

The contraction of $\oF$ in $G$ gives a graph with vertices $V_F \cup
\{w\}$, where $w$ results from the contraction of $\oF$.

\begin{prop}\label{prop:vertexdegreelevelness}
    Let $G=(V,E)$ be a simple, biconnected graph and let $w$ be a vertex such
    that the set of edges $F$ of $G - w$ is a flacet. Then $F$ is a $k$-level
    flacet of $\MM(G)$ if and only if $\deg(w)= k$. 
\end{prop}
\begin{proof}
      Let $E_w$ be the edges incident to $w$.  For a spanning tree $T
      \subseteq E$, we have $\ell_F(\1_T) = |F \cap T| = |T| - |E_w \cap T|$.
      Hence, $F$ is $k$-level if and only if there are at most $k$ spanning
      trees $T_1,\dots,T_k$ such that every $T_i$ uses a different number of
      edges from $E_w$. Since $|E_w|=\deg(w)$ and every spanning tree contains
      at least one edge of $E_w$, there are at most $\deg(w)$ spanning trees
      with different size of the intersection with $F$, thus $k\leq \deg(w)$.
      Moreover, $G$ is simple, thus there exists a spanning tree $T_1$ such
      that $E_w\subseteq T_1$. Applying the same reasoning of the proof of
      Lemma \ref{lem:contract_oF}, we obtain the sequence of spanning trees
      with the desired properties. Finally, we observe that $T_1\cap F$ has
      $\deg(w)-1$ connected components, thus the sequence is made of at least
      $\deg(w)$ trees, proving that $\deg(w)\leq k$.
\end{proof}

It follows from Proposition \ref{prop:vertexdegreelevelness} that the cone
over a biconnected graph on $k$ vertices has a \mbox{$k$-level} flacet.  The next
result gives a strong converse to this observation.  A graph $G$ is called
\Defn{minimally biconnected} if $G\setminus e$ is not biconnected for all $e
\in E$. For more background on this class of graphs we refer to \cite{Plummer}
and \cite{dirac}. 
\begin{prop}\label{prop:pyramidminor}
   Let $G$ be a simple, biconnected graph with a vertex $w$ such that the set
   of edges $F$ not incident to $w$ is a flacet. If $F$ is a $k$-level flacet
   of $\MM(G)$, then $G$ has a minor $\cone(H)$ where $H$ is a minimally
   biconnected graph on $k$ vertices.
\end{prop}
\begin{proof}
    Let $m = |V_F|$. By Proposition \ref{prop:vertexdegreelevelness},
    $\deg(w)=k$ and thus $m\geq k$. By removing edges if necessary, we can
    assume that $F$ is minimally biconnected. By Lemma~\ref{lem:tutte}, the
    contraction of any edge of $F$ leaves a biconnected graph. Contract an
    edge such that at most one endpoint is connected to $w$. The new edge set
    $F^\prime$ is still a $k$-level flacet.  By iterating these
    deletion-contraction steps, we obtain a cone over $F^\prime$ with apex
    $w$.
\end{proof}

\begin{thm}\label{thm:excludminorsgraphs}
    A graph $G$ is $k$-level if and only if $G$ has no minor $\cone(H)$ where
    $H$ is a minimally biconnected graph on $k+1$ vertices.
\end{thm}
\begin{proof}
      Let $G=(V,E)$ be a graph and $F \subset E$ a $m$-level flacet such that $m > k$. By Lemma~\ref{lem:contract_oF}, we may assume that
    $F$ is the set of edges not incident to some $w \in V$. By
    Proposition~\ref{prop:pyramidminor}, we may also assume that $G|_F$ is
    minimally biconnected on $m$ vertices. Now, $G|_F$ contains a minor $H$
    that is minimally biconnected on $k+1$ vertices and hence $G$ contains
    $\cone(H)$ as a minor. 
\end{proof}

\subsection{The class of $3$-level graphs}\label{ssec:3level}

According to Theorem~\ref{thm:excludminorsgraphs}, the excluded minors for
$\GLevels_3$ are cones over minimally biconnected graphs on $4$ vertices. The
only minimally biconnected graph on $4$ vertices is the $4$-cycle
and hence the excluded minor is $W_4 =
\cone(C_4)$, the wheel with $4$ spokes. In general, let us write $W_n =
\cone(C_n)$ for the \Defn{$n$-wheel}, which is a $n$-level graph.
The family of $W_4$-minor-free graphs was considered by R.~Halin
(see~\cite[Ch.~6]{Diestel}). In this section, we will rediscover the
\emph{building blocks} for this class.

We start with the observation that by Lemma~\ref{lem:not3con} and
Corollary~\ref{cor:Series_level}, we may restrict to \mbox{$3$-connected}
graphic matroids. The following proposition allows us to focus on simple
$3$-connected graphs.
\begin{prop}[{\cite[Prop.~1.2]{Oxley5wheel}}]\label{prop:3connequiv}
    If $G$ is a graph with at least $4$ vertices, then $G$ is
    \mbox{$3$-connected} and simple if and only if $\MM(G)$ is $3$-connected.
\end{prop}

In this section we will use higher connectivity of graphs.  Recall that a
graph $G$ is \Defn{$k$-connected} if the removal of any $k-1$ vertices leaves
$G$ connected. In general, graphic matroids of $k$-connected graphs are
not necessarily $k$-connected matroids. However, we are going to consider only
simple graphs with more than 4 vertices, thus by
Proposition~\ref{prop:3connequiv} we do not need to specify if we use
$3$-connectivity in the sense of graphs or in the sense of matroids.
Also, a graph is \Defn{$k$-regular} if every vertex is
incident to exactly $k$ edges.  

\begin{prop}\label{prop:3level_3conn}
    A $3$-level, $3$-connected simple graph is $3$-regular.
\end{prop}
\begin{proof}
    A graph $G$ with a vertex of degree at most $2$ cannot be $3$-connected.
    If there is a vertex $w$ of degree at least $4$, then $G-w$ is
    biconnected. It follows that the set of edges $F$ not incident to $w$
    form a flacet and Proposition~\ref{prop:vertexdegreelevelness} yields the
    claim.
\end{proof}

The following well-known result (see~\cite[Thm~8.8.4]{Oxley}) puts
strong restrictions on minimally $3$-connected matroids. A \Defn{$n$-whirl}
is the matroid of the $n$-wheel $W_n = \cone(C_n)$ with the additional basis
being the rim of the wheel $B = E(C_n)$.

\begin{thm}[Tutte's wheels and whirl theorem]\label{thm:Tutte}
Let $\MM = (E,\BB)$ be a $3$-connected matroid. Then the following are
equivalent:
\begin{enumerate}[\rm (i)]
\item For all $e \in E$ neither $\MM \backslash e$ nor $\MM / e$ is
    $3$-connected;
\item $\MM$ is a $n$-whirl or $n$-wheel, for some $n$.
\end{enumerate}
\end{thm}

We will come back to whirls in the next section. For now, we note that the
only minimally $3$-connected simple graphs are the wheels.  Moreover, note
that every $3$-regular simple graph must have an even number of vertices
($3|V(G)|=2|E(G)|$). 

\begin{lemma}\label{lem:3con_6vert}
Let $G$ be a $3$-connected $3$-regular simple graph with at least $6$
vertices. Then $G$ is at least $4$-level.
\end{lemma}
\begin{proof}
    If $G$ is a wheel, it is easy to see that it is at least $4$-level.
    Suppose now that $G$ is not a wheel. By Theorem \ref{thm:Tutte}, there
    must be an edge $e$ such that $G \setminus e$ or  $G / e$ is
    $3$-connected.  Now, $G\setminus e$ has a degree-2 vertex for all $e \in
    E$ and hence is not $3$-connected.  On the other hand, $\MM(G / e)$ is
    $3$-connected, which implies that $G/e$ is a simple $3$-connected graph.
    In addition, $G/e$ has a vertex of degree $4$.  By
    Proposition~\ref{prop:vertexdegreelevelness}, we conclude that $G / e$
    (and consequently $G$) is at least $4$-level.
\end{proof}

\begin{cor}\label{cor:k4only3level}
    $K_4$ is the only $3$-level, $3$-connected simple graph.
\end{cor}

The following gives a complete characterization of level $3$ graphs.
\begin{thm}\label{thm:level3}
  For a graph $G$ the following are equivalent.
    \begin{enumerate}[\rm (i)]
        \item $G$ has no minor isomorphic to $W_4$;
        \item $G$ is $3$-level;
        \item $G$ can be constructed from the cycles $C_2$, $C_3$, the graph
            $C_3^*$ with $2$ vertices and $3$ parallel edges,  and $K_4$ by
            taking direct or $2$-sums.
    \end{enumerate}
\end{thm}
\begin{proof}
    The wheel $W_4$ is the cone over the $4$-cycle, which is the unique
    minimally biconnected graph on $4$ vertices. The equivalence (i)
    $\Leftrightarrow$ (ii) thus follows from
    Theorem~\ref{thm:excludminorsgraphs}.  For (ii) $\Rightarrow$ (iii), if
    $G$ not $3$-connected, then, by Lemma~\ref{lem:not3con}, $\MM(G)$ can be
    decomposed using direct sums and $2$-sums of $3$-connected graphic
    matroids and we can assume that $G$ is $3$-connected.  If $G$ is
    $3$-level, we are done by Corollary~\ref{cor:k4only3level}. If $G$ is
    $2$-level, the result follows from Theorem~\ref{thm:main1}.
    (iii) $\Rightarrow$ (ii) follows
    from Corollary~\ref{cor:Series_level} and~\eqref{eqn:direct_sum}.
\end{proof}

By inspecting the building blocks for $2$-level
(Example~\ref{ex:serpargraphs}) and $3$-level graphs, it is tempting to think
that the building blocks of $k$-level graphs are given by the building blocks
and the forbidden minors of $\GLevels_{k-1}$.  This turns out to be false even
for $\GLevels_{4}$. Indeed $\Lev(K_5)=4$ and we cannot obtain it as a sequence
of direct sums and $2$-sums of $C_2$, $C_3$, $C_3^*$, $K_4=W_3$, and $W_4$.

\subsection{$4$-level and Theta-$2$ graphs}\label{ssec:4level}

A further hope one could nourish is that $3$-level graphs coincide with the
graphs of Theta rank $2$.  This would be the case if and only if $\Th(W_4) =
3$. The only $k$-level flacet $F$ of $W_n$ with $k > 3$ is given by the rim of
the wheel $F = E(C_n)$. To find a sum-of-squares representation of
$\ell_F(\x)$ for the basis configuration $V_{\MM(W_n)}$ of $W_n$, we may
project onto the coordinates of $F$ which coincides with the configuration of
\emph{forests} of $C_n$. Now, every subset of $E(C_n)$ is independent except
for the complete cycle $I = E(C_n)$.  Hence the configuration of forests is
given by $\{0,1\}^n \setminus \{\1\}$ and the affine function in question is
$\ell(\x) = n-1 - \sum_i x_i$. For $n=4$,
\[
    18 \ell(\x)  \ = \ 2 (\ell(\x)(\ell(\x)-4))^2 + (\ell(\x)(\ell(\x)-1))^2
    \quad\text{ for all } \x \in \{0,1\}^4, \x \neq  \1
\]
gives a sum-of-squares representation~\eqref{eqn:k-sos} of degree $\le 2$.  We
may now pullback the $2$-sos representation to $\ell_F(\x)$ which shows that
$W_4$ is Theta-$2$.

Towards a list of excluded minors for $\GTheta_2$, we focus on the class of
$4$-level graphs. Using Theorem \ref{thm:excludminorsgraphs} we easily find
the two excluded minors for $\GLevels_{ 4}$:
\begin{center}
\includegraphics[scale=3.5]{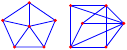}
\end{center}
The first graph is the $5$-wheel $W_5$, the second graph is the cone over
$K_{2,3}$ and is called $A_3\backslash x$ in~\cite{Oxley5wheel}. The next
result states that this is the right class to study.

\begin{prop}\label{prop:W5_theta3}
    The wheel $W_5$ has Theta rank $3$.
\end{prop}
\begin{proof}
    Let $F = E(C_5)$ be the edges of the rim of the wheel which is a flat of
    rank $4$.  This is the unique flacet of levelness $5$ and it is sufficient
    to show that its facet-defining affine function $ \rk(F) - \ell_F(\x) = 4
    - \ell_F(\x)$ is not $2$-sos with respect to the spanning trees $V =
    V_{\MM(W_5)}$ of $W_5$. Arguing by contradiction, let us suppose that
    there are polynomials $h_1(\x),\dots,h_m(\x)$ of degree $\le 2$ such that
    \[
        f(\x) \ := \ 4 - \ell_F(\x) - 
        h_1(\x)^2 - \cdots -h_1(\x)^2 
    \]
    is identically zero on $V$.

    Consider the point $p = \1_F$. This is not a basis of $\MM(W_5)$ and a
    polynomial separating $p$ from $V$ is given by $f$. That is, by
    construction $f$ is a polynomial that vanishes on $V$ and $f(p) \le -1
    \neq 0$. Now we may compute a degree-compatible Gr\"obner basis of the
    vanishing ideal $I = I(V)$ using \emph{Macaulay2}~\cite{M2}. That is, we
    compute a Gr\"obner basis with respect to a term order $\preceq$ such that
    $\deg(\x^a) < \deg(\x^b)$ implies $\x^a \prec \x^b$. For such a Gr\"obner
    basis, it holds that a polynomial $f$ of degree $d$ is contained in $I$ if
    and only if it is in the ideal spanned by the Gr\"obner basis elements of
    degree at most $d$.
    Evaluating the elements of the Gr\"obner basis at $p$ shows that the only
    polynomials not vanishing on $p$ are of degree $5$. As $\deg(f) \le 4$
    by construction, this yields a contradiction.
\end{proof}

The proof suggests an interesting connection to Tutte's wheels and whirls
theorem (Theorem~\ref{thm:Tutte}): For $n = 4$ it states that the vanishing
ideal of the $n$-wheel $I(W_n)$ is generated by $I(\mathcal{W}^n)$ and a
unique polynomial of degree $n$. This should be viewed in relation to
Proposition~\ref{prop:cube_gen}: Projecting $V_{\mathcal{W}^n}$ and $V_{W_n}$
onto the coordinates of $F = E(C_n)$ yields $\{0,1\}^n$ and $\{0,1\}^n
\setminus \1$, respectively.

Oxley~\cite{Oxley5wheel} determined that the class of $3$-connected graphs not
having $W_5$ as a minor consists of $17$ individual graphs and $4$ infinite
families. The graph $A_3 \backslash x$ is clearly among these graphs and
is a minor of the $4$ infinite families as well as three further ones. This proves
the following result.
\begin{thm}
    Every $4$-level graph is obtained by direct and $2$-sums of 
    $C_2$, $C_3$, $C_3^*$, and the following $14$ graphs\\
\centerline{
\includegraphics[scale=1.8]{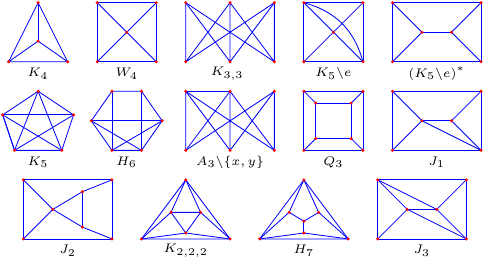}.
}
\qed
\end{thm}

As $A_3 \backslash x$ is Theta-2, a complete list of excluded minor has to be
extracted from the $17$ graphs plus $4$ families in~\cite{Oxley5wheel}. As a
last remark, we note that the Theta-$1$ graphs are given by series-parallel
graphs. The property of being Theta-$2$ however is independent of planarity.

\begin{prop}
    The graphs $K_5$ and $K_{3,3}$ have Theta rank $2$.
\end{prop}
\begin{proof}
    For both cases we use the idea that for a given flacet $F \subseteq E$, we
    may project the basis configuration $V$ onto the coordinates given by $F$
    and find a $2$-sos representation of the affine function $\rk(F) - \sum_i
    x_i$.

    For the graph $K_{3,3}$, the only flacets of levelness $>3$ are given by
    $4$-cycles. Projecting onto these coordinates yields $\{0,1\}^4 \setminus
    \1$ which is a point configuration of Theta rank $2$ as shown at the
    beginning of this subsection.

    For the complete graph $K_5$, we note that the only flacets $F$ of levelness
    $>3$ are given by the edges of an embedded $K_4$. For such a flacet, we
    might equivalently consider $\ell_{E \setminus
    F}(\x)-1 \ge 0$. Projecting onto $E \setminus F$ again yields $\{0,1\}^4 \setminus
    \0$.
\end{proof}

\section{Excluded minors for $k$-level matroids}\label{sec:klevel}

The cone construction (Definition~\ref{dfn:cone}) employed in the previous
section to show the existence of finitely many excluded minors for
$\GLevels_k$ cannot be extended to general matroids. Indeed, whereas any two
trees on $n$ vertices have the same matroid, the matroid of their cones
typically do not. Moreover, there is no Robertson-Seymour theorem for general
matroids: Minor-closed classes of matroids are generally not characterized by
finitely many excluded minors. In this section we show that $k$-level matroids
can be characterized in finite terms and we describe the class explicitly.

A matroid $\MM$ is called \Defn{minimally $\boldsymbol k$-level} if
$\Lev(\MM)=k$ and $\Lev(\NN)<\Lev(\MM)$ for every minor $\NN$ of $\MM$. It is
clear that excluded minors for $\MLevels_k$ are given by the minimally
$l$-level matroids for $l > k$. The main result of this section is the
following.

\begin{thm}\label{thm:finiteexclminorsklevel}
    Excluded minors for the class of $(k-1)$-level matroids are given by the
    minimally $k$-level matroids. Moreover, there are finitely many minimally
    $k$-level matroids.
\end{thm}

Let us formalize a notion that we already saw in the
proof of Lemma~\ref{lem:contract_oF}: A \Defn{$\boldsymbol k$-sequence of
bases} for a flacet $F$ is a collection of bases $B_1,\ldots , B_k \in
\BB(\MM)$ such that
\begin{itemize}[$\circ$]
\item $|F\cap B_1|$ is minimal among all bases of $\MM$,
\item $|F\cap B_{i+1}|=|F\cap B_i|+1$, for $1\leq i < k$,
\item $F\cap B_i\subset F\cap B_{i+1}$, for $1\leq i \leq k{-}1$, and
\item $|F\cap B_k|=\rk_{\MM}(F)$.
\end{itemize}
It is straightforward to verify that $F$ is a $k$-level flacet if and only if
$F$ has a $k$-sequence of bases. Indeed, starting with a basis $B_1$ such that
$|F \cap B_1|$ is minimal, one iteratively alters $B_{i+1}$ by some $e_i \in F
\setminus B_i$. We can also make a more refined choice.

\begin{lemma}\label{lem:ksequencewithe}
    Let $\MM$ be a connected matroid and $F$ a $k$-level flacet of $\MM$. For
    any $e$ in ${\cF := E(\MM) \setminus F}$, there exists a $k$-sequence of bases
    $B_1,\ldots , B_k$ such that $e\in B_i$ for $i=1,\ldots, k$.
\end{lemma}
\begin{proof}
    Since $\MM$ is connected, $e$ is not a loop and we can find a basis $B_1$
    such that $|F \cap B_1|$ is minimal and $e \in B_1$. For $1 \le i < k$,
    $|F \cap B_i| < \rk(F)$. So, there is an $e_i
    \in F \setminus B_i$ such that $(F \cap B_i) \cup \{e_i\}$ is independent.
    Let $C_i \subseteq B_i \cup \{e_i\}$ be the fundamental circuit containing
    $e_i$. Since $F$ is a flat and $C_i$ is not a circuit in $F$, there is
    $f_i \in C_i \setminus (F \cup \{e\})$ and we define
    $B_{i+1} = (B_i \setminus f_i) \cup e_i$.
\end{proof}

Since $\MM_1\oplus_2 \MM_2$ contains both
$\MM_1$ and $\MM_2$ as minors, it follows from Lemma~\ref{lem:not3con} and
Corollary~\ref{cor:Series_level} that every minimally $k$-level matroid is
$3$-connected. 

\begin{prop}\label{prop:minimal_dualisminimalconnected}
    Let $\MM$ be a minimally $k$-level matroid and $F$ a $k$-level flacet of
    $\MM$. Then $(\MM/ F)^*$ is a minimally connected matroid.
\end{prop}
\begin{proof}
    Suppose $(\MM/ F)^*$ is not minimally connected. There exists an element
    $e\in \cF$ such that the deletion $(\MM/ F)^*\setminus e$ is a connected
    matroid.  Since a matroid is connected if and only if its dual is, we
    infer that $(\MM/ F)/ e$ is connected.

    Since $\MM$ is minimally $k$-level, it is $3$-connected. By Lemma
    \ref{lem:ksequencewithe} we can construct a $k$-sequence of bases for $F$
    such that all bases contain $e$. We have that $B_1\setminus e, \ldots ,
    B_k\setminus e$ is a $k$-sequence of bases for $F$ with respect to the
    matroid $\MM/e$. We only need to check that $F$ is a flacet of $\MM/e$.
    If $C$ is a circuit containing $e$ and some elements of $F$, it must
    contain at least a second element $e'\in \cF$ because $F$ is a flat. In
    addition, there must be at least a third element $e''\in \cF$, otherwise
    $e'$ would be a loop of $(\MM/ F)/e$, which is connected by hypothesis.
    This shows that $F$ is a flat of $\MM/e$. Moreover, $(\MM/e)/F\cong
    (\MM/F)/e$ and $(\MM/e)|_F\cong \MM|_F$ are connected. Thus $F$ is a
    $k$-level flacet of $\MM/e$, contradicting the $k$-level minimality of
    $\MM$.
\end{proof}

Similar to the case of graphs, the following proposition states that $\cF =
E(\MM) \setminus F$ is independent for a $k$-level flacet of a minimally
$k$-level matroid.

\begin{prop}\label{prop:minimalklevel_rankoftilde}
    Let $\MM$ be a minimally $k$-level matroid and $F$ a $k$-level flacet of
    $\MM$. Then $\rk(\cF)=|\cF|$.
\end{prop}
\begin{proof}
    By contradiction, suppose $\rk(\cF)<|\cF|$. Consider a $k$-sequence of
    bases $B_1,\ldots , B_k$ for $F$. Because of the assumption
    $\rk(\cF)<|\cF|$, we can pick an element $e\in \cF\setminus B_1$. By
    Proposition \ref{prop:minimal_dualisminimalconnected}, $(\MM/F)/e$ is not
    connected. Since $F$ is a flacet, $\MM/F$ is connected and,
    by Lemma~\ref{lem:tutte}, $(\MM/F)\setminus e$
    is connected.  Now $F$ is a flat of $\MM\setminus e$ and both
    $(\MM\setminus e)|_F\cong \MM|_F$ and $(\MM\setminus e)/F\cong
    (\MM/F)\setminus e$ are connected.  Hence, $F$ is a flacet of the matroid
    $\MM\setminus e$.  The bases $B_1, \ldots, B_k$ are also bases for
    $\MM\setminus e$ and form a $k$-sequence for the flacet $F$. Thus
    $\MM\setminus e$ is a $k$-level minor of $\MM$, contradicting the
    $k$-level minimality of $\MM$.
\end{proof}

\begin{prop}\label{prop:reductionofF}
    Let $\MM$ be a minimally $k$-level matroid and $F$ a $k$-level flacet of
    $\MM$. Then $\MM|_F$ is a minimally connected matroid.
\end{prop}
\begin{proof}
    Suppose that $(\MM|_F)\setminus e$ is connected for some   $e\in F$. Then
    $\Fme=F\setminus e$ is a flat of $\MM\setminus e$. We show that $\Fme$ is
    a $k$-level flacet of $\MM\setminus e$.

    The matroid $(\MM\setminus e)|_{\Fme}\cong (\MM|_F)\setminus e$ is
    connected by hypothesis. Note that $\MM / \Fme$ has $e$ as a loop and
    hence $(\MM\setminus e)/\Fme\cong \MM/F$. Thus, $(\MM\setminus e)/\Fme$ is
    also connected which shows that $\Fme$ is a flacet. 
    
    At last, we show that there is a $k$-sequence of bases of $\MM\setminus e$
    for $\Fme$. Since $\MM|_F$ is connected it has a basis that avoids $e$. We
    can complete this to a basis $B_k$ of $\MM$. Now for $f \in \cF$, $B_k
    \cup f$ contains a circuit and by
    Proposition~\ref{prop:minimalklevel_rankoftilde} this circuit is not
    entirely in $\cF$. Hence, we can define a basis $B_{k-1} := B_k \setminus
    f \cup f'$ for some $f' \in B_k \cap F$. Continuing this way yields a
    $k$-sequence of bases $B_1,\dots,B_k$ for $F$ that avoids $e$ and hence is
    a $k$-sequence of bases for $\Fme$ in $\MM \setminus e$. This contradicts
    the $k$-level minimality of $\MM$.
\end{proof}

\begin{prop}\label{prop:rankofFminklevel}
    Let $\MM$ be a minimally $k$-level matroid and $F$ a $k$-level flacet of
    $\MM$. Then $\rk(F)=k-1$.
\end{prop}
\begin{proof}
    Suppose that $\rk(F)>k{-}1$. Consider a $k$-sequence $B_1, \dots , B_k$
    for $F$: by definition $|F\cap B_k|=\rk(F)>k{-}1$ and thus $|F\cap
    B_1|>0$. Equivalently, there is an element $e\in F$ such that $e\in B_i$
    for $i=1,\ldots , k$. We prove that the matroid $\MM/e$ is $k$-level with
    respect to the flacet $\Fme=F\setminus e$. Since $\MM|_F$ is minimally
    connected by Proposition~\ref{prop:reductionofF}, it follows from
    Lemma~\ref{lem:tutte} that $(\MM/e)|_{\Fme}\cong (\MM|_F)/e$ is connected.
    Also, $(\MM/e)/\Fme\cong \MM/ F$ is connected because $F$ is a flacet of
    $\MM$. Finally, $B_1\setminus e, \dots , B_k\setminus e$ are bases of
    $\MM/e$ and form a $k$-sequence for the flacet $\Fme$, contradicting the
    $k$-level minimality of $\MM$.
\end{proof}

We can finally show that the excluded minors of $\MLevels_k$ are given by the
minimally $(k+1)$-level matroids.

\begin{prop}\label{prop:minorofminklevel}
    Every minimally $(k+1)$-level matroid has a $k$-level minor.
\end{prop}
\begin{proof}
    Let $\MM$ be a minimally $(k+1)$-level matroid and $F$ a $(k+1)$-level
    flacet. Choose a \mbox{$k$-sequence} $B_0,\dots , B_{k}$ for $F$. Pick an
    element $f\in F \setminus B_0$ such that $f \in B_i$ for $i=1,\dots, k$.
    Applying the same reasoning as in the proof of
    Proposition~\ref{prop:rankofFminklevel}, we infer that $\Fme:=F\setminus
    f$ is a flacet of $\MM/f$. Moreover, $B_1\setminus f, \ldots ,B_k\setminus
    f$ is a $k$-sequence of bases which shows that $\MM/f$ is $k$-level.
\end{proof}

To complete the proof of Theorem~\ref{thm:finiteexclminorsklevel}, we show
that for fixed $k$, the size of the ground set of a minimally $k$-level
matroid is bounded. This trivially implies that there only finitely many
minimally $k$-level matroids.  To bound the size of the ground set of a
minimally $k$-level matroid $\MM$, we choose one of its $k$-level flacets $F$
and bound separately the size of $F$ and the size of its complement
$\cF=E(\MM)\setminus F$.  We quote two useful facts from Oxley's book.

\begin{prop}{\cite[Prop.~4.3.11]{Oxley}}\label{prop:finitenessofF}
    Let $\MM$ be a minimally connected matroid of rank $r$ where $r\geq 3$.
    Then $|E(\MM)|\leq 2r{-}2$. Moreover, equality holds if and only if
    $\MM\cong M(K_{2,r-1})$.
\end{prop}

Recall from Section~\ref{ssec:matroids} that a parallel class of a matroid
$\MM$ is a subset $S \subseteq E$ such that for any $e,f \in S$, the set
$\{e,f\}$ is a circuit.

\begin{prop}{\cite[Ch.~4.3, Ex. 10 (d)]{Oxley}} \label{prop:finitenessofG}
    Let $\MM$ be a matroid for which $\MM^*$ is minimally connected. Then
    either $\MM\cong U_{n,1}$ for some $n\geq 3$ or $\MM$ has at least
    $\rk(\MM){+}1$ non-trivial parallel classes.
\end{prop}

\begin{proof}[Proof of Theorem \ref{thm:finiteexclminorsklevel}]
    In light of Theorem~\ref{thm:main1}, we only need to consider $k \ge 3$.
    Let $M$ be a minimally $k$-level matroid $\MM$.  Any $k$-level flacet $F$
    of $\MM$ is of rank $k-1$ by Proposition \ref{prop:rankofFminklevel}; By
    Proposition \ref{prop:reductionofF}, $M|_F$ is minimally connected.  If
    $\rk(F) = 2$, then Proposition~\ref{prop:finitenessofG} implies that $M|_F
    \cong U_{3,2}$. For $\rk(F) \ge 3$, by
    Proposition~\ref{prop:finitenessofF}, $F$ has at most $2(k-1)-2=2k-4$
    elements.
    
    Hence, we need to upper bound the number of elements in $\cF$. 
    Set  
    \[
        T \ := \ \{e \in \cF \; : \; \exists \, C \text{ circuit of $\MM$ 
        with  $e\in C$  and  $|C\cap \cF|=2$} \}.
    \]
    That is, every $e \in T$ is in a non-trivial parallel class in $M / F$.
    The number of non-trivial parallel classes is bounded from above by
    $\frac{|T|}{2}$. Set $S:=\cF \setminus T$.

    Define $h:=\rk(\MM){-}\rk(F){-}1$, so that $\rk(M/F)=h+1$. By
    Proposition~\ref{prop:minimal_dualisminimalconnected}, $(\MM/F)^*$ is
    minimally connected on at least $3$ elements (since for $k \ge 3$ this
    implies $|\cF|\geq 4$). By Proposition~\ref{prop:finitenessofG}
    there are two possibilities:
   
    If $\MM/F \cong U_{|\cF|,1}$, then $\rk(\MM)=k$ and $|\cF|\leq k$ because
    of Proposition~\ref{prop:minimalklevel_rankoftilde}. It follows that
    $|E(\MM)|=|F|{+}|\cF|\leq 2k{-}4{+}k=3k{-}4$. If $k=2$, then $|\cF| \le
    3$.

    On the other hand, if $\rk(\MM/F)=h{+}1>1$, then $\MM/F$ has at least $h +
    2$ non trivial parallel classes. Hence we obtain $|T|\geq 2h + 4$.
    Moreover, by Proposition~\ref{prop:minimalklevel_rankoftilde}, we have
    that 
    ${\cF=\rk(\cF)\leq \rk(\MM)=k+h}$ and this fact yields $|T|\leq k + h$.
    Together this gives 
    \[
        2h{+}4\leq k{+}h \quad \Longrightarrow \quad h\leq k{-}4.
    \]
    It is immediate that $|\cF|\leq k{+}h\leq 2k{-}4$ and thus
    $|E(\MM)|=|F| + |\cF|\leq 2k - 4 {+} 2k - 4=4k - 8.$
\end{proof}

The result of this section does not rule out that matroids of Theta rank~$k$
have infinitely many excluded minors and we did not manange to extend our
techniques to Theta rank. However, we conjecture that the class $\MTheta_k$ of
matroids of Theta rank~$k$ is described by finitely many excluded minors.

\bibliographystyle{myamsalpha}
\bibliography{ThetaLevelMatroid}

\end{document}